\documentclass[10pt]{article}

\usepackage{graphicx}
\usepackage{amssymb}
\usepackage{epsfig}
\usepackage{amsmath}
\usepackage{amsthm}
\usepackage{color}
\definecolor{r}{rgb}{0.9,0.3,0.1}
\definecolor{b}{rgb}{0.1,0.3,0.9}

\usepackage{amsmath}
\usepackage{amssymb}
\usepackage{amsthm}
\usepackage{enumerate}
\usepackage{amsbsy}
\usepackage{amsfonts}
\newtheorem{theo}{Theorem}[section]
\newtheorem{defin}[theo]{Definition}
\newtheorem{prop}[theo]{Proposition}

\newtheorem{lemm}[theo]{Lemma}

\newcommand{\al}{\alpha}
\newcommand{\be}{\beta}
\newcommand{\ga}{\gamma}
\newcommand{\Ga}{\Gamma}

\newcommand{\om}{\omega}

\newcommand{\si}{\sigma}

\newcommand{\te}{\theta}
\newcommand{\De}{\Delta}
\newcommand{\de}{\delta}

\newcommand{\pa}{\partial}

\newcommand{\R}{{\mathbb R}^n}

\newcommand{\ri}{\rightarrow}
\newcommand{\Rn}{{\mathbb R}^{n-1}}

\newcommand{\na}{\nabla}

\newcommand{\calC}{{\mathcal C}}
\newcommand{\calI}{{\mathcal I}}

\newcommand{\bke}[1]{\left( #1 \right)}
\newcommand{\bkt}[1]{\left[ #1 \right]}
\newcommand{\bket}[1]{\left\{ #1 \right\}}
\newcommand{\norm}[1]{\left\Vert #1 \right\Vert}
\newcommand{\abs}[1]{\left| #1 \right|}

\def\XXint#1#2#3{{\setbox0=\hbox{$#1{#2#3}{\int}$}
     \vcenter{\hbox{$#2#3$}}\kern-.5\wd0}}

\newenvironment{proof-stokes}{{\par\noindent\bf
            Proof of Theorem \ref{maintheo-stokes} }}{\hfill\fbox{}\par\vspace{.2cm}}

\newenvironment{proof-stokes-g}{{\par\noindent\bf
            Proof of Theorem \ref{maintheo-stokes-g} }}{\hfill\fbox{}\par\vspace{.2cm}}

\newenvironment{proof-nse-a}{{\par\noindent\bf
            Proof of Theorem \ref{main-nse-est} }}{\hfill\fbox{}\par\vspace{.2cm}}

\newenvironment{proof-nse-b}{{\par\noindent\bf
            Proof of Theorem \ref{main-nse-reg} }}{\hfill\fbox{}\par\vspace{.2cm}}

\newenvironment{prop4}{{\par\noindent\bf
Proof of Proposition \ref{E's}}}{\hfill\fbox{}\par\vspace{.2cm}}

\begin{document}

\title{ Estimates of anisotropic Sobolev spaces with mixed norms for the Stokes system in a half-space}

\author{Tongkeun Chang and Kyungkeun Kang}


\date{}


\maketitle

\begin{abstract}
We are concerned with the non-stationary Stokes system with
non-homogeneous external force and non-zero initial data in
${\mathbb R}^n_+ \times (0,T)$. We obtain new estimates of solutions
including pressure in terms of mixed anisotropic Sobolev spaces. As
an application, some anisotropic Sobolev estimates are presented for
weak solutions of the Navier-Stokes equations in a half-space in
dimension three.


\end{abstract}

\maketitle

\section{Introduction}
\setcounter{equation}{0}

In this paper, we study the non-stationary Stokes system in a
half-space $\R_+ \times [0,T)$, $n\geq 3$
\begin{equation}\label{maineq2}
v_t - \De v + \na p =f, \qquad {\rm div} \, v =0 \qquad \mbox{ in
}\,Q^+_T:=\R_+ \times [0,T),
\end{equation}
where $v:Q^+_T\rightarrow\R$ is the velocity field and
$p:Q^+_T\rightarrow\mathbb R$ is the pressure. We consider the initial
and boundary value problem of \eqref{maineq2}, whereby no slip
boundary conditions are assigned, that is
\begin{equation}\label{maineq2-10}
v(x,0)=v_0(x)\quad \mbox{ and }\quad v(x,t)=0,\quad
x\in\partial\R_+= \Rn.
\end{equation}
The Stokes system is the linearized equations of the Navier-Stokes
equations describing motions of incompressible and viscous fluid
flows, which are given as follows:
\begin{equation}\label{maineq2-20}
v_t - \De v + (v\cdot\nabla) v+\na p =f, \qquad {\rm div} \, v =0.
\end{equation}
Since Leray \cite{L34} proved existence of weak solutions of the
Navier-Stokes equations \eqref{maineq2-20}, regularity of weak
solutions has remained open in three dimensions (see also
\cite{H51}). Lots of contributions have been made so far for
uniqueness and regularity of the Navier-Stokes equations and many
efforts to understand the Stokes system better have been also
performed (see e.g. \cite{P59}, \cite{S62}, \cite{L67}, \cite{VS77},
\cite{FJR}, \cite{CKN}, \cite{G86}, \cite{S88}, \cite{CL00},
\cite{ESS}, \cite{GKT}). However, when boundaries of domains are not
empty, comparatively small number of results have been known,
because of difficulty of pressure up to the boundary (see e.g
\cite{VS82}, \cite{S02}, \cite{Ka}, \cite{GKT1}). Among results with
non-empty boundaries, we recall the following $L^p$ estimates for
Stokes system in half-space (see e.g. \cite{So}): Let $1<p<\infty$
and $f\in L^p(Q^+_T)$ and $v_0\in\dot{B}^{2-\frac{2}{p}}(\R_+)$,
where $\dot{B}^{2-\frac{2}{p}}(\R_+)$ is a homogeneous Besov space.
Then the solution of the Stokes system \eqref{maineq2} and
\eqref{maineq2-10} satisfies
\begin{equation}\label{Lp-estimate}
 \| v_t\|_{L^p(Q^+_T)} + \| \na^2 v\|_{L^p(Q^+_T)} +\| \na
p\|_{L^p(Q^+_T)} \leq c\| f
\|_{L^p(Q^+_T)}+c\|v_0\|_{\dot{B}^{2-\frac{2}{p}}_p (\R_+)}.
\end{equation}
In \cite{GGS}, using the theory of Stokes operator, $v$ satisfies
the following estimate of fractional derivatives for $v_0=0$:
\begin{equation}\label{Lp-frac-est}
\int_0^T\| (\frac{d}{dt})^{1-\alpha}v\|^r_{L^p(\R_+)}dt + \int_0^T\|
A^{1-\alpha}_p v\|^r_{L^p(\R_+)}dt \leq c \int_0^T\| A^{-\alpha}_pf
\|^r_{L^p(\R_+)}dt,
\end{equation}
where $A_p$ is the Stokes operator and $0<\alpha<1$. As a
consequence, when $f={\rm{div}}\,F$ with
$F=(F_{ij})_{i,j=1,\cdots,n}$ and $\alpha=1/2$, \eqref{Lp-frac-est}
yields the following a priori estimate:
\begin{equation}\label{giga-grad-est}
\int_0^T\| (\frac{d}{dt})^{\frac{1}{2}}v\|^r_{L^p(\R_+)}dt +
\int_0^T\| \nabla v\|^r_{L^p(\R_+)}dt \leq c \int_0^T\| F
\|^r_{L^p(\R_+)}dt.
\end{equation}
Estimates for pressure were, however, not given in \cite{GGS}. Koch
and Solonnikov\cite{KS} consider the Stokes system \eqref{maineq2}
with $f=\nabla\cdot F$, zero initial data and no-slip boundary
conditions in three dimensional half space, and established that
\begin{equation}\label{Lp-grad-est}
\| \nabla v\|_{L^p(\mathbb R^3_+\times (0,T))} \leq c \| F
\|_{L^p(\mathbb R^3_+\times (0,T))}.
\end{equation}
Furthermore, it was also shown that there exists a $F\in L^p
(\mathbb R^3_+\times (0,T))$ such that corresponding pressure $p$ of
\eqref{maineq2} is not even in $L^p (\mathbb R^3_+\times (0,T))$
(see \cite[Theorem 1.3]{KS}). For comparison, such result is quite
different to that of the entire space $\mathbb R^3$, whereby $(v,
p)$ of Stokes system satisfies the following estimate:
\begin{equation}\label{Lp-grad-p}
\| \nabla v\|_{L^p(\mathbb R^3\times (0,T))} + \norm{p}_{L^p(\mathbb
R^3\times (0,T))} \leq c \| F \|_{L^p(\mathbb R^3\times (0,T))}.
\end{equation}

The main objective of this paper is to look for relevant function
classes that $f$ and $v_0$ belong to such that control of pressure
similar to \eqref{Lp-grad-p} holds in a half-space (see Theorem
\ref{maintheo-stokes} and Theorem \ref{maintheo-stokes-g}). As an
application of such estimate for the Stokes system, we present a new
estimate especially including the pressure of the Navier-Stokes
equations in a half space (see Theorem \ref{main-nse-est}).

Before we state our main results, we remind some function classes,
which are useful for our purpose. Let $\alpha\in\R$ and
$1<p,q<\infty$. We mean by $H^{\al}_{p} (\R_+ )$ a generalized
Sobolev space in a half-space and we also write a homogeneous
Sobolev space as $\dot H^{\al}_{p} (\R_+)$. Besides, we indicate
that $B^{\al}_{pq} (\R_+)$ and $\dot{B}^{\al}_{pq} (\R_+)$ are usual
Besov space and homogeneous Besov space, respectively. In addition,
we define function spaces $ H^{\al}_{p,0} ({\mathbb R}^n_+)$, $\dot
H^\al_{p,0} ({\mathbb R}^n_+)$,   $ B^{\al}_{pq,0} ({\mathbb
R}^n_+)$ and $\dot B^\al_{pq,0} ({\mathbb R}^n_+)$ as the closures
of $C^\infty_c ({\mathbb R}^n_+)$ in $ H^{\al}_{p} ({\mathbb
R}^n_+)$, $\dot H^\al_{p} ({\mathbb R}^n_+)$, $ B^{\al}_{pq}
({\mathbb R}^n_+)$ and $\dot B^\al_{pq} ({\mathbb R}^n_+)$. We
denote some anisotropic Sobolev spaces by  $H^{\al, \frac12
\al}_{pq} (\R_+ \times (0,T))$ and $H^{\al, \frac12 \al}_{pq,0}
(\R_+ \times (0,T))$ (see Section 2 for more details).

Now we are ready to state our first main result.

\begin{theo}\label{maintheo-stokes}
Let $0\leq \al \leq 2$, $0<T\le\infty$ and $1<p,q<\infty$. Suppose
that $f\in H^{\al-2, \frac12 \al-1}_{pq,0} ({\mathbb R}^n_+ \times
(0,T)).$ and $v_0\in {B}^{\al-\frac{2}{q}}_{pq,0}(\R_+)$ with
${\rm{div}}\, f =0$ and ${\rm{div}}\, v_0 =0$ in the sense of
distributions. Then, there exists a unique solution $v$ in $H^{\al,
\frac12 \al}_{pq} (\R_+ \times (0,T))$
of \eqref{maineq2} such that the following estimate is
satisfied:
\begin{equation}\label{CK-est-100}
\| v\|_{H^{\al, \frac12 \al}_{pq} (\R_+ \times (0,T))} \leq c \|
f\|_{H^{\al-2, \frac12 \al-1}_{pq,0} ({\mathbb R}^n_+ \times
(0,T))}+c \| v_0\|_{ {B}^{\al-\frac{2}{q}}_{pq,0}(\R_+) }.
\end{equation}
Moreover, if $1+1/p<\alpha\leq 2$, then $p\in
L_q(0,T;H_p^{\alpha-1}(\R_+))$ such that
\begin{equation}\label{CK-est-150}
\|p\|_{L_q(0,T;H_p^{\alpha-1}(\R_+))} \leq c\| f\|_{H^{\al-2,
\frac12 \al-1}_{pq,0} ({\mathbb R}^n_+ \times (0,T))}+c \| v_0\|_{
{B}^{\al-\frac{2}{q}}_{pq,0}(\R_+) }.
\end{equation}
If $f \in \dot H^{\al-2, \frac12 \al-1}_{pq,0} ({\mathbb R}^n_+
\times (0,T))$ and $v_0\in \dot{B}^{\al-\frac{2}{q}}_{pq,0}(\R_+)$,
then $v\in \dot H^{\al, \frac12 \al}_{pq} (\R_+ \times (0,T))$
and satisfies
\begin{equation}\label{CK-est-200}
\| v\|_{\dot H^{\al, \frac12 \al}_{pq} (\R_+ \times (0,T))} \leq c
\| f\|_{\dot H^{\al-2, \frac12 \al-1}_{pq,0} ({\mathbb R}^n_+ \times
(0,T))}+c \| v_0\|_{ \dot{B}^{\al-\frac{2}{q}}_{pq,0}(\R_+) }.
\end{equation}
Furthermore, if $1+1/p<\alpha\leq 2$, then $p\in L_q(0,T;\dot
H_p^{\alpha-1}(\R_+))$ such that
\begin{equation}\label{CK-est-250}
\|p\|_{L_q(0,T;\dot H_p^{\alpha-1}(\R_+))}  \leq c \| f\|_{\dot
H^{\al-2, \frac12 \al-1}_{pq,0} ({\mathbb R}^n_+ \times (0,T))}+c \|
v_0\|_{ \dot{B}^{\al-\frac{2}{q}}_{pq,0}(\R_+) }.
\end{equation}
\end{theo}
\noindent
%
%
\begin{remark}
In Theorem \ref{maintheo-stokes}, the solution $v$ means in the
following distribution sense: For any smooth vector field
$\psi\in\mathcal C^{\infty}_0(\R_+\times (0,T))$ with ${\rm
div}\,\psi=0$ and for any smooth scalar function
$\phi\in\calC^{\infty}_0(\R_+\times (0,T))$ $v$ satisfies
\begin{equation}\label{CK-july24-100}
\int^{T}_0\int_{{\mathbb R}^n_+}v\cdot(-\psi_t-\Delta
\psi)dxdt=0,\qquad \int^{T}_0\int_{{\mathbb R}^n_+}v\cdot\nabla\phi
dxdt=0.
\end{equation}
\end{remark}
\begin{remark}
The divergence free conditions for $f$ and $v_0$ in Theorem
\ref{maintheo-stokes} means that $<f, \nabla \phi>=0$ and $<v_0,
\nabla \phi>=0$ for any $\phi\in\calC^{\infty}_0(\mathbb R^n)$. We
also remark that the constant $c$ in
\eqref{CK-est-200}-\eqref{CK-est-250} is independent of $T$ for the
homogeneous case.
\end{remark}

The following is a consequence of Theorem \ref{maintheo-stokes}:

\begin{theo}\label{maintheo-stokes-g}
Let $1<p,q<\infty$ and $0\le \beta \le 1$. Suppose that
$f=\nabla\cdot F$, where $F=(F_{ij})_{i,j=1,\cdots,n}$ is a tensor
such that $F\in L_q(0,T;H^{\beta}_{p,0}(\R_+))$
and $v_0\in {B}^{1+\beta-\frac{2}{q}}_{pq,0}(\R_+)$ with
${\rm{div}}\, f =0$ and ${\rm{div}}\, v_0 =0$ in the sense of
distributions.
Then, there exists a unique solution $v\in H^{1+\beta,
\frac{1+\beta}{2}}_{pq}(\R_+ \times (0,T))$ of \eqref{maineq2} such
that the following estimate is satisfied:
\begin{equation}\label{CK-est-300}
\| v\|_{H^{1+\beta, \frac{1+\beta}{2}}_{pq}(\R_+ \times (0,T))} \leq
c \| F\|_{L_q(0,T;H^{\beta}_{p,0}(\R_+))}+c \|
v_0\|_{{B}^{1+\beta-\frac{2}{q}}_{pq,0}(\R_+)}.
\end{equation}
Furthermore, if $\beta>1/p$, then $p\in L_q(0,T; H^{\beta}_p(\R_+))$
of \eqref{maineq2} such that
\begin{equation}\label{CK-est-310}
\| p\|_{L_q(0,T; H^{\beta}_p(\R_+))}\leq c \|
F\|_{L_q(0,T;H^{\beta}_{p,0}(\R_+))}+c \|
v_0\|_{{B}^{1+\beta-\frac{2}{q}}_{pq,0}(\R_+)}.
\end{equation}
If $F\in L_q(0,T;\dot H^{\beta}_{p,0}(\R_+)$ and $v_0\in \dot
{B}^{1+\beta-\frac{2}{q}}_{pq,0}(\R_+)$, there exist $v\in \dot
H^{1+\beta, \frac{1+\beta}{2}}_{pq}(\R_+ \times (0,T))$ with
$\beta\ge 0$ and  $p\in L_q(0,T; \dot H^{\beta}_p(\R_+))$ with
$\beta>1/p$ such that the estimates \eqref{CK-est-300} and
\eqref{CK-est-310} are replaced by
\begin{equation}\label{CK-est-350}
\| v\|_{\dot H^{1+\beta, \frac{1+\beta}{2}}_{pq}(\R_+ \times (0,T))}
\leq c \| F\|_{L_q(0,T;\dot H^{\beta}_{p,0}(\R_+))}+c \| v_0\|_{
\dot{B}^{1+\beta-\frac{2}{q}}_{pq,0}(\R_+) }.
\end{equation}
\begin{equation}\label{CK-est-360}
\| p\|_{L_q(0,T; \dot H^{\beta}_p(\R_+))}\leq c \| F\|_{L_q(0,T;\dot
H^{\beta}_{p,0}(\R_+))}+c \| v_0\|_{
\dot{B}^{1+\beta-\frac{2}{q}}_{pq,0}(\R_+) }.
\end{equation}
\end{theo}
Although the proof of Theorem \ref{maintheo-stokes-g} is rather
straightforward, the details will be, for clarity, presented in
section 5.
\\
\\
\begin{remark}
In Theorem \ref{maintheo-stokes-g}, in case that $p<n/\beta$, we
observe due to Sobolev imbedding that
\begin{equation}\label{CK-May30-10}
\| p\|_{L_q(0,T; L_{\tilde{p}}(\R_+))}\leq c \|
F\|_{L_q(0,T;H^{\beta}_{p,0}(\R_+))}+c \|
v_0\|_{{B}^{1+\beta-\frac{2}{q}}_{pq,0}(\R_+)},
\end{equation}
where $1/\tilde{p}=1/p-\beta/n$. Compared the result in \cite{KS} to
ours, $F$ was assumed only in $L^p_{x,t}((0,T)\times\R_+)$ in
\cite{KS}, which allows an example that the pressure is not even in
$L^p_{x,t}((0,T)\times\R_+)$. In contrast, Theorem
\ref{maintheo-stokes-g} shows that if $F$ has a bit better
regularity in spatial variables, the pressure $p$ can be found in
$L^p_{x,t}((0,T)\times\R_+)$ with suitable choice of $q$ and
$\tilde{p}$ in \eqref{CK-May30-10}.\qed
\end{remark}
\\
\\
As an application of the estimates of Stokes system, 
we consider the incompressible Navier-Stokes
equations:
\begin{equation}\label{nse-10-march22}
(NS)\,\,\left\{
\begin{array}{cl}
\displaystyle\frac{\partial v}{\partial t} + (v \cdot \nabla ) v-
\Delta v=-\nabla p, \quad\displaystyle \mbox{div } v=0&\qquad\mbox{
in }\,\, \Omega\times [0,T)
\\
\vspace{-3mm}\\
\displaystyle v(x,0)=v_0(x)&\qquad\mbox{ in }\,\, \Omega,
\end{array}\right.
\end{equation}
where  $v$ and $p$ are the flow velocity and the scalar pressure,
respectively. The initial data satisfy the compatibility condition,
i.e., ${\rm{div }}\,v_0=0$ and no slip boundary condition is imposed
for velocity $v$ at the boundary $\partial\Omega$, namely
\begin{equation}\label{nse-20-march22}
 v(x,t)=0,\qquad x\in \partial\Omega.
\end{equation}
In case that $\Omega=\mathbb{R}^3$, the following a priori estimate
of Calder\'on-Zygmund type is well-known:
\begin{equation}\label{nse-30-march22}
\norm{p}_{L^q(\mathbb{R}^3)}\leq
C\norm{v}^2_{L^{2q}(\mathbb{R}^3)},\qquad 1<q<\infty.
\end{equation}
Therefore, when $\Omega=\mathbb{R}^3$, it is straightforward that
weak solutions of \eqref{nse-10-march22} satisfy
\begin{equation}\label{nse-40-march22}
\norm{p}_{L^q(\mathbb{R}^3\times I)}\leq
C\norm{v}^2_{L^{2q}(\mathbb{R}^3\times I)},\qquad 1<q\leq
\frac{5}{3}.
\end{equation}
Weak solutions are defined in section 5 (see Definition
\ref{weak-def}).
%
%
However, it is not clear whether or not
\eqref{nse-30-march22}-\eqref{nse-40-march22} is valid for the case
that $\Omega$ has non-empty boundaries with no-slip condition
\eqref{nse-20-march22}. Instead, the following estimate is known for
the gradient of pressure (see e.g. \cite{So} and \cite{GS}):
\[
\norm{\nabla p}_{L^{m}(0,T;L^{l}(\mathbb{R}^3_+))}\leq
C\norm{v\nabla v}_{L^{m}(0,T;L^{l}(\mathbb{R}^3_+))}+C\|
v_0\|_{{B}^{1-\frac{2}{q}}_{pq,0}(\mathbb{R}^3_+)}
\]
\begin{equation}\label{nse-ck10-june12}
\leq C=C(\norm{u}_{L^{\infty}(0,T;L^{2}(\mathbb{R}^3_+))},
\norm{\nabla u}_{L^{2}(0,T;L^{2}(\mathbb{R}^3_+))}, \|
v_0\|_{{B}^{1-\frac{2}{q}}_{pq,0}(\mathbb{R}^3_+)}),
\end{equation}
where $3/l+2/m=4$ and $1<m<2$. In the case that
$\Omega=\mathbb{R}^3_+$, however, with the aid of estimates of the
Stokes system in Theorem \ref{maintheo-stokes}, we can improve the
estimate \eqref{nse-ck10-june12} in a half-space.
More precisely, we obtain the following:
\begin{theo}\label{main-nse-est}
Let $1<p\le 3/2$, $1/p<\beta\le 1$ and $1<q<2$ with
$3/p+2/q=3+\beta$. Assume that $v_0\in
{B}^{1+\beta-\frac{2}{q}}_{pq,0}(\mathbb{R}^3_+)$ with ${\rm{div}}\,
v_0 =0$ in the sense of distributions. Suppose that $v$ is a weak
solution of the \eqref{nse-10-march22} and $p$ the associated
pressure. Then,
\begin{equation}\label{nse-120-march22}
\| v\|_{L^q(0,T; H^{1+\beta}_p(\mathbb{R}^3_+))} +\| p\|_{L^q(0,T;
H^{\beta}_p(\mathbb{R}^3_+))}\le c \|
\abs{v}^2\|_{L^q(0,T;H^{\beta}_{p,0}(\mathbb{R}^3_+))}+c \|
v_0\|_{{B}^{1+\beta-\frac{2}{q}}_{pq,0}(\mathbb{R}^3_+)}.
\end{equation}
Furthermore, due to Sobolev imbedding,
\begin{equation}\label{nse-121-march22}
\| p\|_{L^q(0,T; L^{\tilde{p}}(\mathbb{R}^3_+))}\leq c \|
\abs{v}^2\|_{L^q(0,T;H^{\beta}_{p,0}(\mathbb{R}^3_+))}+c \|
v_0\|_{{B}^{1+\beta-\frac{2}{q}}_{pq,0}(\mathbb{R}^3_+)},
\end{equation}
where $1/\tilde{p}=1/p-\beta/3$ and $3/\tilde{p}+2/q=3$. The
righthand sides in \eqref{nse-120-march22} and
\eqref{nse-121-march22} are bounded by
$C=C(\norm{u}_{L^{\infty}(0,T;L^{2}(\mathbb{R}^3_+))}, \norm{\nabla
u}_{L^{2}(0,T;L^{2}(\mathbb{R}^3_+))}, \|
v_0\|_{{B}^{1-\frac{2}{q}}_{pq,0}(\mathbb{R}^3_+)})$.
\end{theo}
The proof of Theorem \ref{main-nse-est} will be given in section 5.
This paper is organized as follows. In section 2, we introduce some
function spaces. Theorem \ref{maintheo-stokes} will be proved in
section 3 and section 4 by treating non-homogeneous case and
non-zero initial data separately. Section 5 is devoted to proofs of
Theorem \ref{maintheo-stokes-g} and Theorem \ref{main-nse-est}.

\section{Preliminaries}

In this section, we recall some function spaces, remind some known
results used later and provide proofs of preliminary results useful
for our purpose. We start with introducing function spaces. As a
notational convenience, if there exists a constant $c>0$ such that
$cA\leq B\leq c^{-1}A$, then we write $A\approx B$.

\subsection{Function spaces}\label{spaces} \setcounter{equation}{0}


$\bullet$\,\,(Sobolev and Besov spaces in  ${\mathbb R}^{n}$)\,\,
For $\al \in {\mathbb R}$, we consider  distributions $g_\al$ and
$G_\al$, whose Fourier transforms in ${\mathbb R}^{n}$ are given as
follows: For $\xi  \in {\mathbb R}^n$
\begin{align*}
\widehat{G_\al} (\xi) =(1 + 4\pi^2 |\xi|^2)^{-\frac{\al}{2}},\qquad
  \widehat{ g_{\al}} (\xi) = ( 4\pi^2 |\xi|^2)^{-\frac{\al}{2}}.
\end{align*}
For $\al \in {\mathbb R}$ and $ 1 \leq p \leq \infty $ we define the
generalized Sobolev space $H^{\al }_p ({\mathbb R}^{n})$ and the
homogeneous Sobolev space $\dot H^{\al }_p ({\mathbb R}^{n})$ by
\begin{align*}
H^{\al}_p({\mathbb R}^{n}) &= \{ f \in {\mathcal S}'({\mathbb
R}^{n}) \, | \,\,\|f\|_{ H^{\al}_p({\mathbb R}^{n})}:=\| G_{-\al} *
f
\|_{L^p({\mathbb R}^{n})} <\infty \},\\
\dot H^{\al}_p({\mathbb R}^{n}) &  = \{ f \in {\mathcal S}'({\mathbb
R}^{n}) \, | \,\|f\|_{\dot H^{\al}_p({\mathbb R}^{n})} : = \|
g_{-\al} * f \|_{L^p({\mathbb R}^{n})}<\infty\},
\end{align*}
 where
$*$ is a convolution in ${\mathbb R}^{n}$ and ${\mathcal
S}^{'}({\mathbb R}^{n})$
 is the dual space of the Schwartz space
${\mathcal S}({\mathbb R}^{n})$.
We note that for non-negative integer $k$ and $1 \le p \le \infty$,
\begin{align}\label{sepa-sobolev}
H^{k}_p ({\mathbb R}^{n})& = \{ f \in L^p ({\mathbb R}^{n}) \, | \,
\sum_{0 \leq l \leq k}  D_x^{l} f \in L^p ({\mathbb R}^{n})  \},\\
\dot H^{ k}_p ({\mathbb R}^{n}) &  = \{ f \in L_{{\rm loc}}^p
({\mathbb R}^{n}) \, | \, \sum_{  l = k} D_x^{l} f \in L^p ({\mathbb
R}^{n}) \}.
\end{align}
Let $\al
\in {\mathbb R}$ and $1\leq p, \, q<\infty$. The Besov space and the
homogeneous Besov space in ${\mathbb R}^n$, denoted by
$B_{pq}^\al({\mathbb R}^n)$ and $\dot B_{pq}^\al({\mathbb R}^n)$,
are defined as follows, respectively:
\begin{align*}
B_{pq}^\al({\mathbb R}^n) &= \{ f \,| \, \| f
\|_{B_{pq}^\al({\mathbb
R}^n)} := \| f * \psi\|_{L^p} + [ \sum_{1 \leq k < \infty} (2^{\al k} \| f * \phi_k\|_{L^p})^q ]^\frac1q < \infty \},\\
\dot B_{pq}^\al({\mathbb R}^n) & = \{ f \,| \, \| f \|_{ \dot
B_{pq}^\al({\mathbb R}^n)}  :
       =  [ \sum_{ -\infty < k < \infty} (2^{\al k} \| f * \phi_k\|_{L^p})^q ]^\frac1q < \infty \},
\end{align*}
where $\psi$ and $\phi$ are functions in Schwartz space in ${\mathbb
R}^n$ such that $\phi_k(x) = 2^{kn} \phi(2^k x)$ with ${\rm supp} \,
\hat \phi(\xi)= \{ \xi \, | \, 2^{-2} < |\xi | < 2^2 \}$ and
\[
\hat \psi(\xi) + \sum_{ 1 \leq  k < \infty } \hat \phi(2^{-k}\xi )
=1 \,\,\mbox{ for }\,\xi \in {\mathbb R}^n,  \qquad    \sum_{
-\infty < k < \infty } \hat \phi(2^{-k}\xi ) =1 \,\,\mbox{ for }\,
\xi  \neq 0.
\]
\\
\\
The following properties of complex interpolation and norm
equivalence can be found in \cite[Theorem 6.3.2, Theorem 6.4.5]{BL}
and \cite[Theorem 4.17]{BS}.
\begin{prop}\label{prop2}
Let $1 < p, \, q < \infty$ and $\al_1, \al, \al_2 \in {\mathbb R}$
with $\al_1\le \al\le \al_2$.  
\begin{itemize}
\item[(i)]
 $\dot H^\al_p ({\mathbb R}^{n}) \cap L^p ({\mathbb
R}^{n}) = H^\al_p ({\mathbb R}^{n})$ and $\| f\|_{H^\al_p} \approx
\| f\|_{L^p} + \| f\|_{\dot H^\al_p}$ for $\al > 0$.
\item[(ii)] (Complex interpolation)\,\,Let $\theta\in (0,1)$ with $\al = (1-\te) \al_0 + \te
\al_1$.
\begin{align*}
[ H^{\al_0}_p ,  H^{\al_1}_p ]_{\te} = H^{\al}_{p}, \qquad\qquad [
\dot H^{\al_0}_p ,  \dot H^{\al_1}_p ]_{\te} =
 \dot H^{\al}_{p},\\
 [ B^{\al_0}_{pq} ,  B^{\al_1}_{pq} ]_{\te} =
B^{\al}_{pq}, \qquad\qquad [ \dot B^{\al_0}_{pq} ,  \dot
B^{\al_1}_{pq} ]_{\te} =
 \dot B^{\al}_{pq}.
\end{align*}
\item[(iii)]
Let $k < \al< k+1$ for non-negative integer $k$. Then,
\[
\| f \|_{B^\al_{pq}(\R)}\thickapprox \| f\|_{H^k_p (\R)} +
\sum_{|\be|=k}\Big( \int_{\R} ( \int_{\R} \frac{|D^\be f(x) -D^\be
f(y)|^p}{|x -y|^{n + \frac{p}{q} \al}} dy)^{\frac{q}{p}}
dx\Big)^{\frac1q},
\]
\[
\| f \|_{\dot B^\al_{pq}(\R)} \thickapprox  \sum_{|\be|=k}\Big(
\int_{\R} ( \int_{\R} \frac{|D^\be f(x) -D^\be f(y)|^p}{|x -y|^{n +
\frac{p}{q} \al}} dy)^{\frac{q}{p}} dx\Big)^{\frac1q},
\]
where $\be = (\be_1, \be_2, \cdots, \be_n) \in   ({\mathbb N} \cup \{ 0 \})^n$ and $D^\be = D^{\be_1}_{x_1} D^{\be_2}_{x_2} \cdots D^{\be_n}_{x_n}$.
\end{itemize}
\end{prop}

\noindent $\bullet$\,\,(Sobolev space, Besov space and their dual
spaces in $\R_+$)\,\,
Let $r f$ be a restriction over ${\mathbb R}^n_+$ of the function
$f$ defined in ${\mathbb R}^n$.  For $\al \geq 0$, we define
function spaces in a half space as follows:
\begin{align*}
H^{\al}_p ({\mathbb R}^n_+)
 : = \bket{r f \, | \, f \in H^{\al }_p({\mathbb
 R}^n)},\qquad
 \dot H^{\al}_p ({\mathbb R}^n_+)
 : = \bket{ r f \, | \, f \in \dot H^{\al }_p({\mathbb R}^n)},\\
 B^{\al}_{pq} ({\mathbb R}^n_+)
 : = \bket{ r f \, | \, f \in B^{\al }_{pq}({\mathbb
 R}^n)},\qquad
 \dot B^{\al}_{pq} ({\mathbb R}^n_+)
 : = \bket{ r f \, | \, f \in \dot B^{\al }_{pq}({\mathbb R}^n)}
\end{align*}
with their norms $\| f\|_{X ({\mathbb R}^n_+) }: = \inf \| F\|_{X
({\mathbb R}^n)  }$ for all $F\in X({\mathbb R}^n)$ with $r F=f$,
where $X = H^{\al}_p,\, \dot H^{\al}_p,\, B^{\al}_{pq}$ and $\dot
B^{\al}_{pq}$ . Here we note that for non-negative integer $k$ and
$1 < p < \infty$ (see \cite[Chapter 2]{JK})
\begin{align}\label{CK10-april04}
H^{k}_p ({\mathbb R}^n_+) & =  \{ f \, | \, \sum_{ 0 \leq   l \leq k }
\|  D_x^{l} f\|_{L^p({\mathbb R}^n_+)} < \infty \},\\
\dot H^{ k}_p ({\mathbb R}^n_+) & = \{ f \, | \, \sum_{ l = k } \|
D_x^{l} f\|_{L^p({\mathbb R}^n_+)} < \infty \}.
\end{align}
%
Following arguments in \cite[Chapter 2]{JK}, for $ \al > 0$ and
integers $k_1, k_2\geq 0$ we observe that
\begin{align}\label{CK20-april04}
\begin{array}{l}\vspace{2mm}
\bkt{H^{k_1}_p ({\mathbb R}^n_+), H^{ k_2}_p ({\mathbb R}^n_+)}_\te
= H^{\al}_p ({\mathbb R}^n_+),\quad \bkt{\dot H^{ k_1}_p ({\mathbb
R}^n_+), \dot H^{ k_2}_p ({\mathbb R}^n_+)}_\te = \dot H^{\al}_p
({\mathbb R}^n_+),
\end{array}
\end{align}
where $\al = \te k_1 + (1 -\te) k_2$. For later use, we recall an
interpolation result in \cite[section 1.18.4]{T}.
%

\begin{prop}\label{prop-1}
Let $I=(0,T)$, $ \al_1, \,\, \al_2 \in {\mathbb R}, \,\, 0 < \te< 1,
\,\, \al =
\te \al_1 + (1 -\te) \al_2$ and $1 < p <\infty$. 
\[
\bkt{L^p(I;X_1),
L^p(I;X_2)}_{\theta}=L^p(I;\bkt{X_1;X_2}_{\theta}),
\]
where $(X_1, X_2):=(H^{\al_1 }_p ({\mathbb R}^n_+), H^{\al_2 }_p
({\mathbb R}^n_+))$ or $(X_1, X_2):=(\dot H^{\al_1 }_p ({\mathbb
R}^n_+), \dot H^{\al_2 }_p ({\mathbb R}^n_+))$.
\end{prop}
%

We denote by $H_{p,0}^{-\al}({\mathbb R}^n_+) $, $\dot
H_{p,0}^{-\al} ({\mathbb R}^n_+)$,  $B_{pq,0}^{-\al}({\mathbb
R}^n_+) $ and $\dot B_{pq,0}^{-\al} ({\mathbb R}^n_+)$ the dual
spaces of $ H_{p'}^\al({\mathbb R}^n_+)$, $\dot H_{p'}^\al({\mathbb
R}_+^{n})$, $ B_{p'q'}^\al({\mathbb R}^n_+)$ and $\dot
B_{p'q'}^\al({\mathbb R}_+^{n})$ with $\frac1{p} + \frac1{p'} =1$
and $\frac1q + \frac1{q'} =1$, namely
\begin{align*}
H_{p,0}^{-\al}({\mathbb R}^n_+) =  (H_{p'}^\al({\mathbb
R}^n_+))',\qquad \dot H_{p,0}^{-\al}({\mathbb R}^n_+) =(\dot
H_{p'}^\al({\mathbb R}_+^{n}))',\\
 B_{pq,0}^{-\al}({\mathbb
R}^n_+) = (B_{p'q'}^\al({\mathbb R}^n_+))',\qquad \dot
B_{pq,0}^{-\al}({\mathbb R}^n_+) =(\dot B_{p'q'}^\al({\mathbb
R}_+^{n}))'.
\end{align*}
Here we recall some results of trace theorem between Sobolev spaces
and Besov spaces.

\begin{prop}\label{prop-2}
\begin{itemize}
\item[(i)]
If $w \in H^\al_p ({\mathbb R}^n_+)$ with $ \al > \frac1p$, then
$w|_{x_n =0} \in B^{\al -\frac1p}_p({\mathbb R}^{n-1})$  and
\begin{align*}
\norm{ w\big|_{x_n =0}}_{ B^{\al -\frac1p}_p({\mathbb R}^{n-1})}
\leq c \| w \|_{ H^{\al}_p({\mathbb R}^n_+)}.
\end{align*}
\item[(ii)]
If $w \in \dot H^\al_p ({\mathbb R}^n_+)$ with $ \al > \frac1p$,
then $w\big|_{x_n =0} \in \dot B^{\al -\frac1p}_p({\mathbb
R}^{n-1})$ and
\begin{equation}\label{CK15-april11}
\norm{ w\big|_{x_n =0}}_{ \dot B^{\al -\frac1p}_p({\mathbb
R}^{n-1})} \leq c \| w \|_{ \dot H^{\al}_p({\mathbb R}^n_+)}.
\end{equation}
If $w \in \dot H^\al_p ({\mathbb R}^n_+)$ with $\al \geq 0$ and $w$
is harmonic in ${\mathbb R}^n_+$, then $w\big|_{x_n =0} \in \dot
B^{\al -\frac1p}_p({\mathbb R}^{n-1})$ and
\begin{equation}\label{CK300-april10}
\norm{ w\big|_{x_n =0}}_{\dot B^{\al-\frac1p}_p({\mathbb
R}^{n-1})}\approx \| w\|_{\dot H^\al_p({\mathbb R}^n_+)}.
\end{equation}
\end{itemize}
\end{prop}
We remark that the non-homogeneous case with $\alpha>1/p$ was
treated in \cite{JW} and the estimate \eqref{CK15-april11} of
homogeneous case was shown in \cite{JA}. When $w$ is harmonic and
$\al=0$, \eqref{CK300-april10} was proved in \cite[Lemma 2.1]{KS}.
Via the argument of interpolations, we obtain \eqref{CK300-april10}
for general case $\al>0$.
\\
\\
$\bullet$\,\,(Mixed anisotropic Sobolev spaces)\,\,
Let $0 < \si \leq 1$. For $ f \in C^\infty_0({\mathbb R})$, we
introduce a functional given as follows:
\begin{align*}
I_\si f(t) = \frac{1}{\Ga(\si)} \int_{-\infty}^t
\frac{f(s)}{(t-s)^{1-\si}} ds,\qquad \qquad -\infty < t <\infty,
\end{align*}
where $\Ga(\si)$ is  a Gamma function. The $\sigma-$th fractional
derivative of $f$, $D^\si_t f$, can be defined by
\begin{align*}
D^\si_t f(t) = \frac{d}{dt} I_{1 -\si} f(t)\quad \mbox{for} \quad 0
< \si<1 \quad \mbox{and} \quad D^1_t f(t) = \frac{d}{dt} f(t).
\end{align*}
By the simple computation, we have
\begin{align}\label{sum}
I_{\si_1}(I_{\si_2} f) = I_{\si_1 + \si_2} f, \qquad
 D^{\si_1}_t( D_t^{\si_2} f) = D^{\si_1 + \si_2}_t f
\end{align}
for $0 < \si_1, \, \si_2 < 1$,  $\si_1 + \si_2 \leq 1$ and $f \in
C^\infty_0 ({\mathbb R})$.
Furthermore, we have
\begin{equation}\label{CK-july21-100}
\widehat{I_{\si} f} (\tau) =  m_{\si}(\tau) \hat f(\tau),
\qquad\qquad m_{\si}(\tau) =  (a_{\si} - i b_{\si} sign (\tau))
|\tau|^{-\si},
\end{equation}
where
\begin{align*}
a_\si = \int_0^\infty \frac{\cos \, t}{t^\si} dt, \qquad\qquad
b_\si = \int_0^\infty \frac{\sin \, t}{t^\si} dt.
\end{align*}
Therefore, we note that
\[
\widehat{D^\si_t f} (\tau) = i \tau m_{1-\si}(\tau) \hat f(\tau) = i
(a_{1-\si} - i b_{1-\si} sign (\tau)) sign (\tau) |\tau|^{\si} \hat
f(\tau).
\]
We also define the dual operator $J_\si$ of $I_\si$ as follows:
\begin{align*}
J_\si f(t) =\frac{1}{\Ga (\si)}  \int_t^\infty \frac{f(s)}{(s -t)^{1
-\si}} ds, \qquad D^{* \si }_t f(t) = D_t J_{1-\si} f(t)
\end{align*}
such that
\begin{align*}
\int_{{\mathbb R}} I_\si f(t) g(t) dt = \int_{{\mathbb R}} f(t)
J_\si g(t) dt, \qquad \int_{{\mathbb R}} D_t^\si f(t) g(t) dt = -
\int_{{\mathbb R}} f(t)D^{* \si }_t g(t) dt.
\end{align*}
Using the Plancherel theorem, we get $\widehat{J_\si g}(\tau) =\bar
m_\si(\tau) \hat g(\tau)$, where $\bar m_\si$ is the  complex
conjugate of $m_\si$ in \eqref{CK-july21-100}. Rewriting $\bar m_\si
(\tau) = m_\si (\tau) \frac{\bar m_\si (\tau)}{m_\si (\tau)}$ and
noting that
\begin{align*}
\frac{\bar m_\si (\tau)}{m_\si (\tau)} =\frac{a_\si^2 -b_\si^2 +2 i
a_\si b_\si sign (\tau)  }{a_\si^2 + b_\si^2},
\end{align*}
we get
\begin{align*}
J_\si g(t) = I_\si (A_\si g + B_\si H g)(t), \quad 
\quad  D_t^{\si *} g(t) = D_t^\si  (A_{1-\si} g + B_{1-\si} H g)(t),
\end{align*}
where $A _\si = (a_\si^2 -b_\si^2)/(a_\si^2 + b_\si^2)$, $B_\si = (2
a_\si b_\si)/(a_\si^2 + b_\si^2)$ and $H$ is the Hilbert transform.
Next, we define the mixed anisotropic Sobolev spaces $ H^{\al,
\frac12 \al}_{pq}({\mathbb R}^{n+1})$ and $\dot H^{\al, \frac12
\al}_{pq}({\mathbb R}^{n+1})$ with $0 < \al \leq 2$ as follows:
\begin{align*}
 H^{\al, \frac12 \al}_{pq}({\mathbb R}^{n+1}) : & = \{ f \in L^q_t L^p_x \, | \, \| f\|^q_{ H^{\al, \frac12 \al}_{pq}({\mathbb R}^{n+1}) }:
            = \int_{{\mathbb R}}  \Big(\| D^{\frac12 \al}_t f\|^q_{L^p({\mathbb R}^n)}
                    + \| f\|^q_{H^\al_p({\mathbb R}^n)} \Big) dt < \infty \},\\
\dot H^{\al, \frac12 \al}_{pq}({\mathbb R}^{n+1}) : & = \{ f \in L^1_{loc}({\mathbb R}^{n+1}) \, | \,
                    \| f\|^q_{\dot H^{\al, \frac12 \al}_{pq}({\mathbb R}^{n+1}) }:
            = \int_{{\mathbb R}}  \Big(\| D_t^{\frac12 \al} f\|^q_{L^p({\mathbb R}^n)}
               + \| f\|^q_{\dot H^\al_p({\mathbb R}^n)} \Big) dt < \infty \}.
\end{align*}
Note that if  $ {\rm supp} \, f \subset [0,T)$ (so that $f(T) =0$),
then
\begin{align*}
J_\si f(t) =\frac{1}{\Ga (\si)}  \int_t^T \frac{f(s)}{(s -t)^{1
-\si}} ds.
\end{align*}
Let $I = (0,T)$ and we denote $Q_T = \R \times I$, $Q^+_T = \R_+
\times I$, unless any confusion is to be expected. Here we define
the mixed anisotropic Sobolev spaces $ H^{\al, \frac12
\al}_{pq,*}(Q_T^+)$ and $\dot H^{\al, \frac12 \al}_{pq,*}(Q_T^+),
\,\, 0 < \al \leq 2$ in the following way: We say that $f$ is in  $
H^{\al, \frac12 \al}_{pq,*}(Q_T^+)$  ($f$ is in $\dot H^{\al,
\frac12 \al}_{pq,*}(Q_T^+)$) means ${\rm supp} \, f \subset
\overline{\R_+} \times [0, T)$ and

\begin{align*}
\| f\|^q_{ H^{\al, \frac12 \al}_{pq,*}(Q^+_T) }: &
            = \int_{I}  \Big(\| D^{*\frac12 \al}_t f\|^q_{L^p({\mathbb R}^n_+)}
                    + \| f\|^q_{H^\al_p({\mathbb R}^n_+)} \Big) dt < \infty\\
\Big(\| f\|^q_{\dot H^{\al, \frac12 \al}_{pq,*}(Q^+_T) }: &
            = \int_{I}  \Big(\| D^{*\frac12 \al}_t f\|^q_{L^p({\mathbb R}^n_+)}
                    + \| f\|^q_{\dot  H^\al_p({\mathbb R}^n_+)} \Big) dt < \infty \Big).
\end{align*}
We mean by $ H^{-\al , -\frac12 \al}_{p'q',0} (Q_T^+)$ and $ \dot
H^{-\al , -\frac12 \al}_{p'q',0} (Q_T^+)$ the dual spaces of  $
H^{\al , \frac12 \al}_{pq,*} (Q_T^+)$ and $ \dot H^{\al , \frac12
\al}_{pq,*} (Q_T^+)$, where $\frac1p + \frac{1}{p'} =1$ and $\frac1q
+ \frac{1}{q'} =1$. Similarly, we define function spaces $H^{\al,
\frac12 \al}_{pq,*}(Q_T)$,  $\dot H^{\al, \frac12 \al}_{pq,*}(Q_T) $
and their dual spaces $H^{-\al, -\frac12 \al}_{p'q',0}(Q_T)$, $ \dot
H^{-\al, -\frac12 \al}_{p'q',0}(Q_T)$, respectively.

\begin{remark}
We note that $L^q(I; C^\infty_c(\R))$ is dense subset of $H^{-\al,
-\frac12 \al}_{pq,0}(Q_T)$ and $\dot H^{-\al, -\frac12
\al}_{pq,0}(Q_T)$. We also remark that  $\dot H^{\al, \frac12
\al}_{pq,*} (\R_+ \times I) \subset L^q (I;\dot H^{\al}_p (\R_+))$,
$L^q (I;\dot H^{-\al}_p (\R_+)) \subset \dot H^{-\al , -\frac12
\al}_{pq,0} (\R_+ \times I) $ and
\begin{align*}
\| f \|_{\dot H^{-\al , -\frac12 \al}_{pq,0} (\R_+ \times I)} \leq
\| f \|_{L^q_t (I;\dot H^{-\al}_p (\R_+))}.
\end{align*}
\end{remark}

For $\phi \in \dot H^{\al,\frac12 \al}_{pq*}(Q_T)$ ($\phi \in
H^{\al,\frac12 \al}_{pq*}(Q_T))$, we define $\tilde \phi \in \dot
H^{\al, \frac12 \al}_{pq}({\mathbb R}^{n+1})$ ($\tilde \phi \in
H^{\al, \frac12 \al}_{pq}({\mathbb R}^{n+1})$ ) by
\begin{align}\label{CK-july22-100}
\tilde \phi(x,t) = \left\{ \begin{array}{ll}
\phi(x,t), & x \in {\mathbb R}^n, \,\,  0 < t <T,\\
\phi(x,-t), & x \in {\mathbb R}^n, \,\,  -T < t < 0,\\
0, & \mbox{otherwise}
\end{array}
\right.
\end{align}
and it is direct that $\| \tilde\phi\|_{\dot H^{\al, \frac12
\al}_{pq}({\mathbb R}^{n+1})  } \leq c \| \phi\|_{\dot H^{\al,
\frac12 \al}_{pq*}(Q_T)}$ ( $\| \tilde\phi\|_{ H^{\al, \frac12
\al}_{pq}({\mathbb R}^{n+1})  } \leq c \| \phi\|_{ H^{\al, \frac12
\al}_{pq*}(Q_T)}$ ). Next proposition shows that fractional
derivatives are bounded operators between anisotropic Sobolev
spaces.

\begin{prop}\label{iterated}
Let $0 \le \si_1, \, \si_2 \le 1$ and  $\si_1 + \si_2 \leq 1$.
\begin{itemize}
\item[(i)]
$D_t^{\frac12\si_1}$ and $\De^{\frac12\si_1}$ are bounded operators
from $\dot H^{\si_1+\si_2, \frac12 \si_1 + \frac12 \si_2}_{pq}({\mathbb R}^{n+1})$ to
$\dot H^{\si_2, \frac12 \si_2}_{pq}({\mathbb R}^{n+1})$.
\item[(ii)]
$D_t^{* \frac12\si_1}$  and $\De^{\frac12 \si_1}$ are bounded
operators from $\dot H^{\si_1+\si_2, \frac12 \si_1 + \frac12
\si_2}_{pq*}(Q_T)$ to $\dot H^{\si_2, \frac12 \si_2}_{pq*}(Q_T)$.
\end{itemize}
\end{prop}

\begin{proof}
We recall by definition that
\begin{equation}\label{CK-july21-200}
\| D_t^{\frac12 \si_1} f\|^q_{\dot H^{\si_2, \frac12 \si_2}_{pq}} =
\int_{{\mathbb R}} (\| D_t^{\frac12 \si_2} D_t^{\frac12 \si_1} f
\|^q_{L^p({\mathbb R}^n)} +  \| \De^{\frac12 \si_2} D_t^{\frac12
\si_1} f \|^q_{L^p({\mathbb R}^n)}) dt.
\end{equation}
Since $D_t^{\frac12 \si_2} D_t^{\frac12 \si_1} f = D_t^{\frac12
\si_1  +\frac12 \si_2} f $, the first term in \eqref{CK-july21-200}
is dominated by $\| f\|^q_{\dot H^{\si_1 + \si_2, \frac12 \si_1 +
\frac12 \si_2}_{pq}} $. On the other hand, we note that
\[
{\mathcal F}(  \De^{\frac12 \si_2} D_t^{\si_1} f)(\xi, \tau) = i\tau
|\xi|^{\si_2}  m_{1-\si_1}(\tau) \hat f(\xi, \tau)
\]
\begin{equation}\label{CK-july21-300}
=\frac{ i\tau |\xi|^{\si_2} m_{1-\si_1}(\tau) }{ |\xi|^{\si_1 +
\si_2}  + i\tau m_{1-\si_1 - \si_2}(\tau)} ( |\xi|^{\si_1 + \si_2} +
i\tau m_{1-\si_1 - \si_2}( \tau)) \hat f.
\end{equation}
Here, ${\mathcal F}$ is the Fourier transform in ${\mathbb R}^{n+1}$.
We note that $\frac{i\tau |\xi|^{\si_2} m_{1-\si_1}(\tau) }{
|\xi|^{\si_1 + \si_2} +i\tau m_{1-\si_1 - \si_2}( \tau)} $ is a
$L^q_tL^p_x$-multiplier (see e.g. \cite[ Theorem A]{Liz}), and thus
the term in \eqref{CK-july21-300} is dominated by $\| f\|^q_{\dot
H^{\si_1 + \si_2, \frac12 \si_1 + \frac12 \si_2}_{pq}} $. Hence, $
D^{\frac12 \si_1}$ is a bounded operators from $\dot H^{\si_1+\si_2,
\frac12 \si_1 + \frac12 \si_2}_{pq}$ to $\dot H^{\si_2, \frac12
\si_2}_{pq}$. Similarly, we can see that $ \De^{\frac12 \si_1}$
bounded operators from $\dot H^{\si_1+\si_2, \frac12 \si_1 + \frac12
\si_2}_{pq}$ to $\dot H^{\si_2, \frac12 \si_2}_{pq}$. This completes
the proof of $(i)$.

For $\phi \in \dot H^{\si_1+\si_2, \frac12 \si_1 + \frac12
\si_2}_{pq*}(Q_T)$, we let $\tilde \phi  \in \dot H^{\si_1+\si_2,
\frac12 \si_1 + \frac12 \si_2}_{pq}({\mathbb R}^{n+1})$ as defined
in \eqref{CK-july22-100}. Then, 
via the result (i), we obtain
\[
\|D_t^{* \frac12\si_1} \tilde \phi\|_{ \dot H^{ \si_2,  \frac12
\si_2}_{pq}({\mathbb R}^{n+1}) } \leq c \|\tilde \phi\|_{ \dot
H^{\si_1+\si_2, \frac12 \si_1 + \frac12 \si_2}_{pq}({\mathbb
R}^{n+1})}  \leq c \|  \phi\|_{ \dot H^{\si_1+\si_2, \frac12 \si_1 +
\frac12 \si_2}_{pq*}(Q_T)}.
\]
Since $D_t^{* \frac12\si_1} \tilde \phi (x,t) = D_t^{* \frac12\si_1}
\phi (x,t)$ for $(x,t) \in Q_T$, we have
\begin{align*}
\|D_t^{* \frac12\si_1} \phi\|_{ \dot H^{ \si_2,  \frac12
\si_2}_{pq*}(Q_T) }  \leq c \|  \phi\|_{ \dot H^{\si_1+\si_2,
\frac12 \si_1 + \frac12 \si_2}_{pq*}(Q_T)}.
\end{align*}
This complete the proof of $(ii)$ and thus we deduce the
proposition.
\end{proof}


\begin{remark}
With the aid of duality argument, we also conclude that  $D_t^{
\frac12\si_1}$ and $\De^{\frac12 \si_1}$ are bounded operators from
$\dot H^{-\si_2, -\frac12 \si_2}_{pq,0}(Q_T)$ to $\dot
H^{-\si_1-\si_2, -\frac12 \si_1 - \frac12 \si_2}_{pq,0}(Q_T)$.
\end{remark}

\begin{prop}\label{prop-3}
Let $ 0 \leq \al \leq 2$,  $f \in \dot H^{\al -2, \frac12 \al
-1}_{pq,0} (Q_T )$ and
$Tf(x,t): = < f, \Ga(x-\cdot, t-\cdot )>$.
Then, $Tf \in \dot H^{\al, \frac12 \al}_{pq} (Q_T )$ and the
following estimate holds:
\begin{align*} 
\| Tf\|_{\dot H^{\al, \frac12 \al}_{pq} (Q_T )}
  & \leq c  \|  f\|_{\dot H^{\al-2, \frac12 \al-1}_{pq,0} (Q_T )}.
\end{align*}
\end{prop}
\begin{proof}
For the case of $f \in L^{pq}_{x,t}(Q_T)$, it was shown in \cite{MS}
that
$Tf \in \dot H^{2, 1}_{pq,0} (Q_T )$ and
\begin{align*}
\int_0^T \bke{\|D^2_x Tf(\cdot,t)\|^q_{L^p({\mathbb R}^n)}   + \|D_t
Tf(\cdot,t)\|^q_{L^p({\mathbb R}^n)}}dt  \leq c \int_0^T
               \|f(\cdot,t)\|^q_{L^p({\mathbb R}^n)} dt.
\end{align*}
With the same argument, we define a bounded operator $T^*: L^{pq}
(Q_T) \ri \dot H^{2, 1}_{pq*} (Q_T )$ by
\begin{align*}
T^* g(y,s): = \int_s^T \int_{{\mathbb R}^n} \Ga(x-y, t-s) g(x,t)
dxdt.
\end{align*}
Since $\dot H^{-2, -1}_{pq,0} (Q_T)  = ( \dot H^{2, 1}_{p'q'*} (Q_T
) )^*$, for $f \in  L^{q}(I; \R)$, we have via H\"older's equality
\[
\| Tf\|_{L^{pq}(Q_T)} = \sup_{ \|g \|_{ L^{p'q'}(Q_T)} =1} \int \int_{Q_T}
Tf(x,t) g(x,t) dxdt
\]
\[
= \sup_{\|g \|_{ L^{p'q'}(Q_T)} =1} \int  \int_{Q_T} f(x,t) T^* g(x,t)
dxdt
\]
\[
\leq  \sup_{\|g \|_{ L^{p'q'}(Q_T)} =1} \|T^*g\|_{\dot H^{2, 1}_{p'q'*}
(Q_T)}  \|f\|_{ H^{-2, -1}_{pq,0} (Q_T)}
\]
\[
\leq c \sup_{\|g \|_{ L^{p'q'}(Q_T)} =1} \|g\|_{L^{p'q'} (Q_T)}  \|f\|_{
H^{-2, -1}_{pq,0} (Q_T)}\leq  c  \|f\|_{ H^{-2, -1}_{pq,0} (Q_T)}.
\]
%
Since $L^q(I; \R)$ is dense in $\dot H^{-2, -1}_{pq,0} (Q_T)$, it
follows that for all $f \in  \dot H^{-2, -1}_{pq,0} (Q_T)$
\[
\| Tf\|_{L^{pq}(Q_T)} \leq  c  \|f\|_{ H^{-2, -1}_{pq,0} (Q_T)}.
\]
Noting that $D^{\frac12\al}_t T f = T D^{\frac12 \al}_t f$ and
$\De^{\frac12\al} T f = T \De^{\frac12 \al} f$,  by Proposition \ref{iterated}, we get
\[
\int_0^T \| D^{\frac12 \al}_t  Tf \|^q_{L_x^p({\mathbb R}^n)} dt
\leq c  \| D^{\frac12 \al}_t f\|^q_{H^{-2, -1}_{pq,0}(Q_T)}\leq  c
\|  f\|^q_{H^{ \al-2,\frac12 \al -1}_{pq,0}(Q_T)},
\]
\[
\int_0^T \|  \De^{\frac12 \al}  Tf \|^q_{L^p({\mathbb R}^n)} dt\leq
c  \| \De^{\frac12 \al} f\|^q_{H^{-2, -1}_{pq,0}(Q_T)}\leq  c  \|
f\|^q_{H^{ \al-2,\frac12 \al -1}_{pq,0} (Q_T)}.
\]
This completes the proof.
%
\end{proof}


\begin{remark}
For $f \in \dot H^{-\al,-\frac12 \al}_{pq,0}(Q^+_T)$ ($f \in
H^{-\al,-\frac12 \al}_{pq,0}(Q^+_T))$, we define $\tilde f \in \dot
H^{-\al, -\frac12 \al}_{pq,0}(Q_T)$ ($\tilde f \in H^{-\al, -\frac12
\al}_{pq,0}(Q_T)$ ) by
\begin{align*}
< \tilde f, \phi>  =< f, \phi|_{Q^+_T}>, \quad \phi \in \dot
H^{-\al, -\frac12 \al}_{pq,0}(Q_T) ( \phi \in H^{-\al, -\frac12
\al}_{pq,0}(Q_T) )
\end{align*}
and it is immediate that $\| \tilde f\|_{\dot H^{\al, \frac12
\al}_{pq}(Q_T) } \leq c \| f\|_{\dot H^{\al, \frac12
\al}_{pq*}(Q^+_T)}$
 ( $\| \tilde f\|_{ H^{\al, \frac12 \al}_{pq}(Q_T)  } \leq c \| f\|_{ H^{\al, \frac12 \al}_{pq*}(Q^+_T)}$ ).
Therefore, by proposition \ref{prop-3}, we have
\begin{align*}
\| T \tilde f\|_{\dot H^{\al, \frac12 \al}_{pq} (Q_T^+)} \leq  \| T\tilde f \|_{\dot H^{\al, \frac12 \al}_{pq} (Q_T)} \leq c \| \tilde f \|_{\dot H^{\al -2, \frac12 \al -1}_{pq, 0} (Q_T)} \leq c  \|  f \|_{\dot H^{\al -2, \frac12 \al -1}_{pq, 0} (Q^+_T)}
\end{align*}
for $f \in \dot H^{\al -2, \frac12 \al -1}_{pq, 0} (Q^+_T)$.
\end{remark}


\section{
Stokes system with $f\neq 0$ and $v_0=0$} \label{kernal}
\setcounter{equation}{0}

In this section, we consider the Stokes system with nonhomogeneous
external force and zero initial data in ${\mathbb R}^n_+\times
(0,\infty)$, namely
\begin{equation}\label{C10K-april7}
v_t - \De v + \na p =f, \qquad {\rm div} \, v =0 \qquad \mbox{ in
}\,\,Q^+_T:=\R_+ \times [0,T)
\end{equation}
\begin{equation}\label{C20K-april7}
v(x,0)=0\quad \mbox{ and }\quad v(x,t)=0,\quad x\in\partial\R_+=\Rn.
\end{equation}
As mentioned before, we then study the homogeneous Stokes system
non-zero initial data, i.e. $f=0$ and $v_0\neq 0$ in a half space
next section, and by combining both cases we can obtain the desired
estimates, since Stokes system is linear.
For simplicity, we denote by $E(x)$ and $\Ga(x,t)$ the fundamental
solutions of the Laplace equation and the heat equation,
respectively, that is,
\begin{align*}
E(x) = \frac1{2\pi} \log \, |x| \quad \mbox{if}\,\,\,
n=2,\qquad\quad E(x) = - \frac1{(n-2)\om_n |x|^{n-2}} \quad
\mbox{if} \,\,\, n > 2,
\end{align*}
where $\om_n$ is the area of the unit sphere in ${\mathbb R}^n$ and
\begin{align*}
\Ga(x,t) &= (4\pi t)^{-\frac{n}{2}} e^{-\frac{|x|^2}{4t}} \chi_{t >
0}.
\end{align*}
We recall that it was shown in \cite{So} that if the boundary data $
f$ is in $ L^q(I; C^\infty_c(\R_+))$, then the system
\eqref{C10K-april7}-\eqref{C20K-april7} has the following solution
formulae:
\begin{equation}\label{expression-v}
v (x,t) =\int_0^t \int_{{\mathbb R}^n_+} {\mathcal G}(x,y, t-\tau)
f(y,\tau) dyd\tau,
\end{equation}
\begin{equation}\label{expression-p}
p(x,t) =\int_0^t \int_{{\mathbb R}^n_+} {\mathcal P}(x,y, t-\tau)
\cdot f(y,\tau)dyd\tau,
\end{equation}
where the matrix ${\mathcal G} = (G_{ij})_{1 \leq i,j \leq n} $ and
the vector ${\mathcal P} = (P_i)_{1 \leq i \leq n}$ are given as
\begin{equation}\label{formulas-v}
G_{ij}= \de_{ij} (\Ga(x-y, t) - \Ga(x-y^*,t )) + 4(1 -\de_{jn})
\frac{\pa}{\pa x_j} \int_0^{x_n} \int_{{\mathbb R}^{n-1}}
            \frac{\pa E(x-z)}{\pa x_i} \Ga(z -y^* , t) dz,
\end{equation}
\[
P_j(x,y,t)=4 (1 - \de_{jn}) \frac{\pa }{\pa x_j}\Big[ \int_{{\mathbb
R}^{n-1}} \frac{\pa E(x' - z', x_n)}{\pa x_n} \Ga(z' -y', y_n,t) dz'
\]
\begin{equation}\label{formulas-p}
+\int_{{\mathbb R}^{n-1}} E(x' -z',x_n) \frac{\pa \Ga(z'-y', y_n,
t)}{\pa y_n} dz'\Big].
\end{equation}
Here $x'$ stands for $(x_1, \cdots, x_{n-1})$ and $x^* = (x', -x_n)$
indicates the point symmetric to $x$ with respect to the plane $x_n
= 0$. Since $C_c^\infty(Q_T^+)$ is dense subset of $H^{\al-2,
-\frac12 \al -1}_{pq,0}(Q_T^+)$, we may assume that $f $ is in
$C_c^\infty(Q_T^+)$ such that the solution $(v,p)$ of system
\eqref{C10K-april7}-\eqref{C20K-april7} are represented by
\eqref{formulas-v} and \eqref{formulas-p}. Using the formula
\eqref{expression-v} of $v$ and formula \eqref{expression-p} of $p$
, we will prove the following estimates:
\begin{equation}\label{C30K-april7}
\| v\|_{H^{\al, \frac12 \al}_{pq} (\R_+ \times (0,T))} + \| p\|_{L^q
((0,T); H^{\al-1}_{p} (\R_+  ))}  \leq c \| f\|_{H^{\al-2, \frac12
\al-1}_{pq,0} ({\mathbb R}^n_+ \times (0,T))},
\end{equation}
\begin{equation}\label{C40K-april7}
\| v\|_{\dot H^{\al, \frac12 \al}_{pq} (\R_+ \times (0,T))} +\|
p\|_{L^q ((0,T); \dot H^{\al-1}_{p} (\R_+)) }\leq c \| f\|_{\dot
H^{\al-2, \frac12 \al-1}_{pq,0} ({\mathbb R}^n_+ \times (0,T))}.
\end{equation}
Since the way of proofs for above estimates are similar, we consider
only the case of \eqref{C40K-april7}. We start with the estimate of
velocity field.

\subsection{Estimates of velocity fields}\label{velocity}
For convenience, for a measurable function $g$ in $\mathbb R^n_+$ we
define
\begin{equation}\label{C50K-april7}
\calI g(x',x_n) := \int_{{\mathbb R}^{n-1}} E(x' -y', 0) g(y',x_n)
dy'.
\end{equation}
From the formula \eqref{expression-v} and the functional
\eqref{C50K-april7}, we decompose $v_i$ as follows:
\begin{equation}\label{C60K-april7}
v_i(x,t) = u_i(x,t) + 4
    \frac{\pa }{\pa x_i} w(x,t)
         -4\de_{in} \calI \sum_{1 \leq j \leq n-1} \frac{\pa}{\pa x_j}  {\mathcal U}^*
         f_j(x,t),\qquad i=1,\cdots,n,
\end{equation}
where $u_i$, $w$ and ${\mathcal U}^* f_j$ are defined by
\begin{equation}\label{C70K-april7}
u_i(x,t) :=\int_0^t \int_{{\mathbb R}^n_+} (\Ga(x-y, t-\tau)
-\Ga(x-y^*, t-\tau) )f_i(y, \tau) dy d\tau,
\end{equation}
\begin{equation}\label{C80K-april7}
w(x,t) := \int_0^{x_n} \int_{{\mathbb R}^{n-1}} E(x-y) \sum_{1 \leq
j \leq n-1} \frac{\pa}{\pa y_j} {\mathcal U}^* f_j(y,t) dy,
\end{equation}
\begin{equation}\label{C90K-april7}
{\mathcal U}^* f_j(x,t) := \int_0^t \int_{{\mathbb R}^n_+}\Ga(x-y^*,
t-\tau) f_j(y, \tau) dy d\tau.
\end{equation}

We consider separately the above terms and first estimate $u_i$ for
$i=1,\cdots,n$.

\begin{lemm}\label{lemma-1}
Let $0\leq \al \leq 2$ and $0<T<\infty$. Suppose that $f\in H^{\al
-2, \frac\al2 -1}_{pq,0}(\R_+ \times (0,T))$ with ${\rm{div}}\, f
=0$ in the sense of distributions. If $u_i$ is given in
\eqref{C70K-april7} for $i=1,\cdots,n$, then
\begin{align*} 
\| u_i\|_{\dot H^{\al, \frac12 \al}_{pq} ({\mathbb R}^n_+ \times
(0,T))} & \leq c  \|  f_i\|_{\dot H^{\al-2, \frac12 \al-1}_{pq,0}
({\mathbb R}^n_+ \times (0,T))}.
\end{align*}
\end{lemm}
\begin{proof}
Let $I=(0,T)$. We note that since $L^q(I; C^\infty_c ({\mathbb
R}^n_+ ))$ is dense subspace of $L^q(I;\dot H^{\al -2}_{p,0}
({\mathbb R}_+^n ))$, we assume without loss of generality that $f
\in L^q(I; C^\infty_c ({\mathbb R}^n_+ )) $ with ${\rm div} \, f
=0$, and we then perform a priori estimates. Let   $f \in L^q(I;
C^\infty_c ({\mathbb R}^n_+ ))$ and $\tilde f$ be a zero extension
of $f$. Then, for $(x,t) \in {\mathbb R}^n_+ \times (0,T)$
\begin{align*}
u_i(x,t) &=\int_0^t \int_{{\mathbb R}^n} \Big( \Ga(x-y, t-\tau) -
\Ga(x-y^*, t-\tau) \Big)
              \tilde f_i(y, \tau) dy d\tau\\
            &=\int_0^t \int_{{\mathbb R}^n} \Ga(x-y, t-\tau) \Big( \tilde f_i(y', y_n, \tau)
                            - \tilde f^*_i(y', y_n, \tau) \Big)
                             dy d\tau
   := T(\tilde f - \tilde f^*),
\end{align*}
where $ f^*(y) = f(y', -y_n)$. Hence, from proposition \ref{prop-3},
we have
\[
\| u_i\|_{\dot H^{\al, \frac12 \al}_{pq} ({\mathbb R}^n_+ \times
(0,T))}=\| T ( \tilde f_i -\tilde f^*_i)\|_{\dot H^{\al,
\frac12 \al}_{pq} ({\mathbb R}^n \times (0,T))}
\leq c  \| \tilde f_i -\tilde f^*_i\|_{\dot H^{\al-2, \frac12
\al-1}_{pq,0} ({\mathbb R}^n \times (0,T))}
\]
\begin{equation}\label{U-*}
\leq c  \| \tilde f_i \|_{\dot H^{\al-2, \frac12 \al-1}_{pq,0}
({\mathbb R}^n \times (0,T))}\leq c  \|  f_i\|_{\dot H^{\al-2,
\frac12 \al-1}_{pq,0} ({\mathbb R}^n_+ \times (0,T))}.
\end{equation}
%
This
completes the proof.
\end{proof}

\begin{lemm}\label{lemma-2}
Let the assumption in Lemma \ref{lemma-1} hold. If $w$ is given in
\eqref{C80K-april7}, then for $1 \leq i \leq n$
\begin{align}\label{v-w}
\|\frac{\pa }{\pa x_i} w\|_{\dot H^{\al,\frac12 \al}_{pq}({\mathbb
R}^n_+\times (0,T))} \leq c \sum_{1 \leq j \leq n-1} \| f_j\|_{\dot
H^{\al-2,\frac12 \al-1}_{pq,0}({\mathbb R}^n_+\times (0,T))}.
\end{align}
\end{lemm}
\begin{proof}
We note that $w=0$ on $\{x_n=0\}$ and $w$ solves for each $t$
\begin{equation}\label{CK100-april7}
\De w(x,t)= -\frac12 \sum_{1 \leq j \leq n-1} \frac{\pa}{\pa x_j}
{\mathcal U}^* f_j(x,t)
      + \frac{\pa }{\pa x_n}  \calI \sum_{1 \leq j \leq n-1} \frac{\pa}{\pa x_j} {\mathcal U}^* f_j(x,t)
      \qquad \mbox{ in }\,\,{\mathbb R}^n_+.
\end{equation}
We denote by $R^{'}=(R^{'}_1,\cdots, R^{'}_{n-1})$  the Riesz
transforms in ${\mathbb R}^{n-1}$ and we set
\[
F_j: = -\frac12 {\mathcal U}^*f_j,\qquad j=1\cdots, n-1,
\]
\begin{equation}\label{C100K-april7}
F_n: = \calI\sum_{1 \leq j \leq n-1} \frac{\pa}{\pa x_j} {\mathcal
U}^* f_j(x,t) = \sum_{1 \leq j \leq n-1}R^{'}_j {\mathcal U}^*
f_j(x',x_n,t).
\end{equation}
%
Then, the right-hand side of \eqref{CK100-april7} is equal to
$\rm{div}\, F$. By the representation formula of Poisson problem for
Lapalce equation in ${\mathbb R}^n_+$, we have
\begin{align*}
w(x,t) & = -\int_{{\mathbb R}^n_+} (E(x-y) - E(x-y^*) ) div \, F(y,t) dy\\
        & = \int_{{\mathbb R}^n_+} D_y (E(x-y) - E(x-y^*) ) \cdot F(y,t) dy.
\end{align*}
Next, we show that for $k=1, 2, 3$
\begin{equation}\label{k-w}
\int_0^T\| D_x^kw(\cdot, t)\|^q_{L^p({\mathbb R}^n_+)}  dt\leq c
\sum_{1 \leq j \leq n-1} \int_0^T \| {\mathcal U}^* f_j(\cdot,
t)\|^q_{\dot H^{k-1}_p({\mathbb R}^n_+)} dt.
\end{equation}
Indeed, in case that $k=1, 2$, we observe that
\[
\int_0^T\| D_x^kw(\cdot, t)\|^q_{L^p({\mathbb R}^n_+)}  dt\leq c
\int_0^T   \sum_{1 \leq j \leq n} \|\na^{k-1}F_j(\cdot,
t)\|^q_{L^p({\mathbb R}^n_+)}
\]
\begin{equation}\label{CK25-april11}
\leq c \sum_{1 \leq j \leq n-1} \int_0^T \|D_x^{k-1} {\mathcal U}^*
f_j(\cdot, t)\|^q_{L^p({\mathbb R}^n_+)} dt\leq c \sum_{1 \leq j
\leq n-1} \int_0^T \| {\mathcal U}^* f_j(\cdot, t)\|^q_{\dot
H^{k-1}_p({\mathbb R}^n_+)} dt.
\end{equation}
When $k=3$ and $(l,m,i)\neq (n,n,n)$, we have
\[
\int_0^T\| D_{lmi}w(\cdot, t)\|^q_{L^p({\mathbb R}^n_+)}  dt\leq c
\int_0^T \sum_{1 \leq j \leq n} \|D_x^{2}F_j(\cdot,
t)\|^q_{L^p({\mathbb R}^n_+)}
\]
\begin{equation}\label{CK35-april11}
\leq c \sum_{1 \leq j \leq n-1} \int_0^T \|D_x^{2} {\mathcal U}^*
f_j(\cdot, t)\|^q_{L^p({\mathbb R}^n_+)} dt\leq c \sum_{1 \leq j
\leq n-1} \int_0^T \| {\mathcal U}^* f_j(\cdot, t)\|^q_{\dot
H^{2}_p({\mathbb R}^n_+)} dt.
\end{equation}
It remains to estimate the case that $(l,m,i)=(n,n,n)$. We note
first that
\begin{align*}
D_{nnn} w(x,t) & = D_{nn} \int_{{\mathbb R}^n_+} D_{y_n}(E(x-y) + E(x-y^*) )\,{\rm div}  \, F(y,t) dy\\
   & = - D_{nn}\int_{{\mathbb R}^n_+}(E(x-y) + E(x-y^*) )   D_{y_n} \,{\rm div} \,  F(y,t) dy\\
    & \quad + D_{nn} \int_{{\mathbb R}^{n-1}} E(x'-y', x_n) \,{\rm div}  \, F(y',0,t) dy': =J_1 + J_2.
\end{align*}
As in the above cases, we estimate the term $J_1$ as follows:
\begin{equation}\label{CK45-april11}
\int_0^T\|  J_1 (\cdot, t)\|^q_{L^p({\mathbb R}^n_+)}  dt \leq
\int_0^T\|  D_{y_n} \,{\rm div} \,  F(\cdot,t)   \|^q_{L^p({\mathbb
R}^n_+)} dt\leq c \sum_{1 \leq j \leq n-1} \int_0^T \| {\mathcal
U}^* f_j(\cdot, t)\|^q_{\dot H^{2}_p({\mathbb R}^n_+)} dt.
\end{equation}
On the other hand, we rewrite $J_2$ as
\[
J_2(x,t) =  D_n \int_{{\mathbb R}^{n-1}} P_{x_n} (x'-y') \,{\rm div}
\, F(y',0,t) dy',
\]
where $P_{x_n}$ is the Poisson Kernel in a half-space. With the aid
of estimates in proposition \ref{prop-2},
\[
\int_0^T\|   J_2 (\cdot, t)\|^q_{L^p({\mathbb R}^n_+)}  dt \leq
\int_0^T\|  \,{\rm div} \,  F(\cdot,0,t)   \|^q_{\dot B^{1
-\frac1p}_p({\mathbb R}^{n-1})}  dt \leq  \int_0^T\|   F(\cdot,0,t)
\|^q_{\dot B^{2 -\frac1p}_p({\mathbb R}^{n-1})}  dt
\]
\begin{equation}\label{CK55-april11}
\leq  \int_0^T\|   F(\cdot,t)   \|^q_{\dot H^{2}_p({\mathbb
R}_+^{n})}  dt \leq c \sum_{1 \leq j \leq n-1} \int_0^T \| {\mathcal
U}^* f_j(\cdot, t)\|^q_{\dot H^{2}_p({\mathbb R}^n_+)} dt.
\end{equation}
Summing up \eqref{CK25-april11}-\eqref{CK55-april11}, we obtain
\eqref{k-w}. Using \eqref{k-w} and Proposition \ref{prop-1}, we get
for $i=1, \cdots, n$
\[
\int_0^T\|\frac{\pa }{\pa x_i} w(\cdot, t)\|^q_{\dot
H^\al_p({\mathbb R}^n_+)} dt\leq c \sum_{1 \leq j \leq n-1} \int_0^T
\| {\mathcal U}^* f_j(\cdot, t)\|^q_{\dot H^\al_p({\mathbb R}^n_+)}
dt
\]
\begin{equation}\label{w-sobolev}
\leq c \sum_{1 \leq j \leq n-1} \|  f_j\|^q_{\dot H^{\al-2, \frac12
\al-1}_{pq,0} ({\mathbb R}^n_+ \times (0,T))}.
\end{equation}
For the time regularity of $w$,  we have
\begin{align*}
D^{\frac{\al}2}_t  w(x,t) = \int_0^{x_n} \int_{{\mathbb R}^{n-1}}
E(x-y) \sum_{1 \leq j \leq n-1} \frac{\pa}{\pa y_j}
D^{\frac{\al}2}_t {\mathcal U}^* f_j(y,t) dy.
\end{align*}
Applying the above argument, for $1 \leq i \leq n$  we get
\begin{align}\label{t-w}
\notag \int_0^T\|D^{\frac{\al}2}_t \frac{\pa}{\pa x_i}
w\|^q_{L^p({\mathbb R}^n_+)} dt
& \leq c  \sum_{1 \leq j \leq n-1} \int_0^T   \|  D^{\frac{\al}2}_t  {\mathcal U}^* f_j(x,\cdot)\|^q_{L^p ({\mathbb R}^n_+)} dx  \\
& \leq c \sum_{1 \leq j \leq n-1} \|  f_j\|^q_{\dot H^{\al-2,\frac12
\al -1}_{pq,0}}.
\end{align}
Combining \eqref{w-sobolev} and  \eqref{t-w}, we obtain \eqref{v-w}.
This completes the proof.
\end{proof}

\begin{lemm}\label{lemma-3}
Let the assumption in Lemma \ref{lemma-1} hold. If ${\mathcal U}^*
f_j$ is given in \eqref{C90K-april7}, then
\begin{align}\label{v-3}
\|   \calI \sum_{1 \leq j \leq n-1} \frac{\pa }{\pa x_j} {\mathcal
U}^* f_j(\cdot,t) \|_{\dot H^{\al,\frac12 \al}_{pq} ({\mathbb R}^n_+
\times I)} \leq c  \sum_{1 \leq j \leq n-1}\|  f_j \|_{\dot
H^{\al-2,\frac12 \al-1}_{pq,0} ({\mathbb R}^n_+ \times I)}.
\end{align}
\end{lemm}
\begin{proof}
Since for $1 \leq j \leq n-1$,  $ D_x^2 \calI  \frac{\pa }{\pa x_j}
{\mathcal U}^* f_j(x,t)=  \calI
        \frac{\pa }{\pa x_j} D_x^2  {\mathcal U}^* f_j(x,t)  $,
recalling \eqref{C100K-april7}, we have
\begin{align}\label{3-2-v}
\int_0^T \| D_x^2 \calI \sum_{1 \leq j \leq n-1} \frac{\pa }{\pa
x_j} {\mathcal U}^* f_j(\cdot,t) \|^q_{L^p ({\mathbb R}^n_+)} dt &
\leq c \sum_{1 \leq j \leq n-1} \int_0^T \|  D_x^2  {\mathcal U}^*
f_j(\cdot,t) \|^q_{L^p ({\mathbb R}^n_+)} dt,
\end{align}
\begin{align}\label{3-0-v}
\int_0^T \|  \calI \sum_{1 \leq j \leq n-1} \frac{\pa }{\pa x_j}
{\mathcal U}^* f_j(\cdot,t) \|^q_{L^p ({\mathbb R}^n_+)} dt & \leq c
\sum_{1 \leq j \leq n-1} \int_0^T \|  {\mathcal U}^* f_j(\cdot,t)
\|^q_{L^p ({\mathbb R}^n_+)} dt.
\end{align}
Via \eqref{3-2-v}, \eqref{3-0-v},  Proposition \ref{prop-1} and
Lemma \ref{lemma-1}, we obtain
\begin{align}
\int_0^T \|  \calI \sum_{1 \leq j \leq n-1} \frac{\pa }{\pa x_j}
{\mathcal U}^* f_j(\cdot,t) \|^q_{\dot H^{\al}_p ({\mathbb R}^n_+)}
dt & \leq c \sum_{1 \leq j \leq n-1} \int_0^T \|  {\mathcal U}^*
f_j(\cdot,t) \|^q_{\dot H^{\al }_{p} ({\mathbb R}^n_+)} dt\\
& \leq c \sum_{1 \leq j \leq n-1} \| f_j\|^q_{ \dot H^{\al
-2,\frac12 \al -1}_{pq,0} ({\mathbb R}^n_+ \times (0,T))}.
\end{align}
For the time regularity, note that $1 \leq j \leq n-1$,
$D^{\frac{\al}2}_t \calI  \frac{\pa }{\pa x_j} {\mathcal U}^*
f_j(x,t)= \calI  \frac{\pa }{\pa x_j} D^{\frac{\al}2}_t {\mathcal
U}^* f_j(x,t) $, recalling \eqref{C100K-april7} and Lemma
\ref{lemma-1}, we have
\begin{align}\label{3-2-v-06-12}
 \int_0^T \| D^{\frac{\al}2}_t \calI \sum_{1 \leq j \leq n-1}
\frac{\pa }{\pa x_j} {\mathcal U}^* f_j(\cdot,t) \|^q_{L^p ({\mathbb
R}^n_+)} dt & \leq c \sum_{1 \leq j \leq n-1}
        \int_0^T \| D^{\frac{\al}2}_t  {\mathcal U}^* f_j(\cdot,t) \|^q_{L^p ({\mathbb R}^n_+)} dt\\
& \leq c\sum_{1 \leq j \leq n-1} \| f_j\|^q_{ \dot H^{\al -2,\frac12
\al -1}_{pq,0} ({\mathbb R}^n_+ \times (0,T))}.
\end{align}
This completes the proof.
\end{proof}
Therefore, with the aid of Lemma \ref{lemma-1}, Lemma \ref{lemma-2}
and Lemma \ref{lemma-3}, we conclude that
\begin{equation}\label{C110K-april7}
\| v\|_{\dot H^{\al, \frac12 \al}_{pq} (\R_+ \times (0,T))} \leq c
\| f\|_{\dot H^{\al -2, \frac12 \al -1}_{pq,0}(\R_+ \times (0,T)) }.
\end{equation}
Following similar procedure, we can show that
\begin{equation}\label{C200K-april10}
\| v\|_{H^{\al, \frac12 \al}_{pq} (\R_+ \times (0,T))} \leq c \|
f\|_{H^{\al -2, \frac12 \al -1}_{pq,0}(\R_+ \times (0,T)) }.
\end{equation}
Since its verification of \eqref{C200K-april10} is almost the same,
we omit the detail.

\subsection{Estimates of pressure} \label{pressure}
We estimate the pressure term. Recalling the formula
\eqref{expression-p}-\eqref{formulas-p}, we split $p$ in two terms,
i.e. $p(x,t) = p_1 (x,t) + p_2(x,t)$, where
\begin{equation}\label{CK10-april8}
p_1(x,t) =4 \sum_{1 \leq j \leq n-1}\int_{{\mathbb R}^{n-1}}
\frac{\pa^2 E(x'-z',x_n)}{\pa x_j \pa x_n} \int_0^t \int_{{\mathbb
R}^n_+} \Ga(z' -y', y_n, t -\tau) f_j(y, \tau) dyd\tau dz',
\end{equation}
\begin{equation}\label{CK20-april8}
p_2(x,t) =4 \sum_{1 \leq j \leq n-1}\int_{{\mathbb R}^{n-1}}
\frac{\pa E(x'-z',x_n)}{ \pa x_j} \int_0^t \int_{{\mathbb R}^n_+}
\frac{\pa \Ga(z' -y', y_n, t-\tau)}{\pa y_n}f_j(y,\tau) dyd\tau dz'.
\end{equation}

\begin{lemm}\label{lemm-1-1}
Let the assumption in Lemma \ref{lemma-1} hold. If $p_1$ is given in
\eqref{CK10-april8}, then
\begin{align}
\| p_1\|_{ L^q((0,T); \dot H^{\al-1}_p ({\mathbb R}^n_+))} & \notag\leq c   \|  f \|_{\dot H^{\al-2,\frac12
\al-1}_{pq,0} ({\mathbb R}^n_+ \times I)}.
\end{align}
\end{lemm}
\begin{proof}
We note that
\begin{align}\label{pressure-p-1}
p_1(x,t) = 4 \sum_{1 \leq j \leq n-1}  P_{x_n} \frac{\pa}{\pa x_j} U
f_j(x',t),
\end{align}
where $P_{x_n}$ is the Poisson integral of the Laplace equation and
\begin{align*}
U f_j(y',t) & = \int_0^t \int_{{\mathbb R}^n_+}
                  \Ga(z' -y', z_n, t)  f_j(z, \tau) dz d\tau.
\end{align*}
Hence, by the properties of Poisson integral of Laplace equation
(see Proposition \ref{prop-2}), for $ \al \geq 1$,
\begin{align}
\notag \int_0^T \| p_1(\cdot, t)\|^q _{ \dot H^{\al-1}_p ({\mathbb
R}^n_+)} dt
 &\leq c \sum_{1 \leq j \leq n-1} \int_0^T
            \| \frac{\pa}{\pa x_j} U f_j(\cdot, t) \|^q_{ \dot B^{\al-1 -\frac1p}_p ({\mathbb R}^{n-1})} dt\\
 & \leq c \sum_{1 \leq j \leq n-1} \int_0^T
             \| U f_j (\cdot, t)\|^q_{\dot B_p^{\al -\frac1p} ({\mathbb R}^{n-1}) } dt.
\end{align}
Let ${\mathcal U} f_j(y,t) = \int_0^t \int_{{\mathbb R}^n_+}
                 \Ga(z' -y',z_n -y_n, t)  f_j(z, \tau) dz d\tau$
such that ${\mathcal U}f_j|_{y_n =0} = Uf_j$. Using Proposition
\ref{prop-2}  and Proposition \ref{prop-1}, we obtain
\[
\int_0^T \| p_1\|^q_{\dot H_p^{\al-1 }( {\mathbb R}^n_+ )}  dt\leq
\sum_{1 \leq j \leq n-1} \int_0^T \|   U f_j\|^q_{\dot B_p^{ \al
-\frac1p}( {\mathbb R}^{n-1})}dt\leq c \sum_{1 \leq j \leq n-1}
\int_0^T \|   {\mathcal U} f_j\|^q_{\dot H_p^{ \al }( {\mathbb
R}^n_+)}dt
\]
\begin{equation}\label{pressure-sobolev}
\leq c \sum_{1 \leq j \leq n-1} \|  f_j\|^q_{ \dot H_{pq,0}^{
\al-2,\frac12 \al -1}( {\mathbb R}^n_+ \times I ) )}.
\end{equation}
This completes the proof.
\end{proof}

\begin{lemm}\label{lemma-1-2}
Let the assumption in Lemma \ref{lemma-1} hold. If $p_2$ is given in
\eqref{CK20-april8}, then for $\al > 1 + \frac1p$,
\begin{align} \label{pressure-p-2}
  \|  p_2\|_{ L^q((0,T); \dot H^{\al-1}_p ({\mathbb R}^n_+))}
   \leq c  \sum_{1\leq j \leq n-1} \|  f_j \|_{\dot H^{\al-2,\frac12 \al-1}_{pq,0} ({\mathbb R}^n_+ \times I)}.
\end{align}
\end{lemm}
\begin{proof}
As before, we denote by $P_{x_n}$ the Poisson integral of the
Laplace equation in ${\mathbb R}^n_+$ and $R^{'}_j$ indicates Riesz
transform in ${\mathbb R}^{n-1}$. We then rewrite $p_2$ as
\begin{align} \label{pressure-p-22}
p_2(x,t) = 4 \sum_{1 \leq j \leq n-1}  P_{x_n} R^{'}_j   U_2
f_j(x',t),
\end{align}
where
\begin{align*}
U_2 f_j(y',t) & = \int_0^t \int_{{\mathbb R}^n_+}
                  \frac{\pa}{\pa z_n}\Ga(z' -y', z_n, t)  f_j(z, \tau) dz d\tau.
\end{align*}
We note that  $U_2f_j={\mathcal U}_2 f_j|_{y_n =0}$, where
\begin{align*}
{\mathcal U}_2 f_j(y,t)  =- \frac{\pa}{\pa y_n} {\mathcal U}
f_j(y,t)  = \int_0^t \int_{{\mathbb R}^n_+} \frac{\pa}{\pa
z_n}\Ga(z' -y',z_n - y_n, t)  f_j(z, \tau) dz d\tau.
\end{align*}
Via the properties of Poisson integral and Proposition \ref{prop-2},
for $ \al > 1+\frac1p$ we have
\[
\int_0^T \|  p_2(\cdot, t)\|^q _{ \dot H^{\al-1}_p ({\mathbb
R}^n_+)} dt\le c \sum_{1 \leq j \leq n-1} \int_0^T \|   U_2
f_j(\cdot, t) \|^q_{\dot B^{\al-1 -\frac1p}_p ({\mathbb R}^{n-1})}
dt
\]
\[
\leq c \sum_{1 \leq j \leq n-1}\|  f_j \|_{\dot H^{\al-2,\frac12 \al-1}_{pq,0} ({\mathbb R}^n_+ \times I)}.
\]
This completes the proof.
\end{proof}

\begin{remark}
In case $1\le \alpha\le 1+\frac1p$, the pressure $p_2$
is decomposed to $p_2 = Q + \partial_t P$ such
that
\begin{equation}\label{CK-july24-10}
\| Q\|_{ L^q((0,T); \dot H^{\al-1}_p ({\mathbb R}^n_+))} \leq c   \|
f \|_{\dot H^{\al-2,\frac12 \al-1}_{pq,0} ({\mathbb R}^n_+ \times
I)},
\end{equation}
\begin{equation}\label{CK-july24-20}
\| P\|_{L^q (0,T; \dot H^{\al +1}_p(\R_+))}  \leq c \| f\|_{  \dot
H^{\al-2, \frac12 \al -1}_p(\R_+ \times I))}.
\end{equation}
Indeed, since $div \, f =0$, we get
\begin{align}\label{p-2decompose}
\begin{array}{ll}\vspace{2mm}
p_2(x,t) &=4 \sum_{1 \leq j \leq n-1}\int_{{\mathbb R}^{n-1}}
\frac{\pa E(x'-z',x_n)}{ \pa x_j} \int_0^t \int_{{\mathbb R}^n_+}
\frac{\pa \Ga(z' -y', y_n, t-\tau)}{\pa y_n}f_j(y,\tau) dyd\tau
dz'.\\
\vspace{2mm} & =4  \int_{{\mathbb R}^{n-1}}
  E(x'-z',x_n)  \int_0^t \int_{{\mathbb R}^n_+}
\frac{\pa \Ga(z' -y', y_n, t-\tau)}{\pa y_n} D_{y_n}f_n(y,\tau)
dyd\tau dz'\\ \vspace{2mm} &=4  \int_{{\mathbb R}^{n-1}}
  E(x'-z',x_n)  \int_0^t \int_{{\mathbb R}^n_+}
\De' \Ga(z' -y', y_n, t-\tau) f_n(y,\tau) dyd\tau dz'\\
\vspace{2mm} &  +4 D_t \int_{{\mathbb R}^{n-1}}
  E(x'-z',x_n)  \int_0^t \int_{{\mathbb R}^n_+}
 \Ga(z' -y', y_n, t-\tau) f_n(y,\tau) dyd\tau dz'\\
\vspace{2mm} &  := Q + \partial_t P.
\end{array}
\end{align}
The estimate \eqref{CK-july24-10} is  similar to
the lemma \ref{lemm-1-1} and we skip its details. For the second term
$P$, we estimate
\[
\| P\|_{L^q (0,T; \dot H^{\al +1}_p(\R_+))} \leq  c  \| P|_{x_n
=0}\|_{L^q (0,T; \dot B^{\al +1 -\frac1p}_p(\Rn))}\leq c \|  U
f_n\|_{L^q (0,T; \dot B^{\al  -\frac1p}_p(\Rn))}
\]
\begin{equation}\label{CK-july24-30}
\leq  c\| {\mathcal U} f_n\|_{L^q (0,T; \dot H^{\al}_p(\R_+))}\leq  c\| f_n\|_{
\dot H^{\al-2, \frac12 \al -1}_p(\R_+ \times I))}.
\end{equation}
However, the estimate \eqref{CK-july24-30} of $\partial_t P$ is not
available and thus neither is the pressure $p$. This is what was
essentially shown in \cite[Theorem 1.5]{KS} for the case $\alpha=1$.
Our estimates of the pressure above, however, implies that if $f$
has a bit better regularity, i.e. $f\in \dot H^{\al-2,\frac12
\al-1}_{pq,0} ({\mathbb R}^n_+ \times I)$ with $\alpha> 1+\frac1p$,
the pressure $p$ can be controlled in terms of $f$.
\end{remark}

\section{Stokes system with $f=0$ and $v_0\neq 0$}
\setcounter{equation}{0}

In this section, we consider the Stokes system with homogeneous
external force and non-zero initial data in ${\mathbb R}^n_+\times
(0,\infty)$, namely
\begin{equation}\label{C10K-april9}
v_t - \nu\De v + \na p =0, \qquad {\rm div} \, v =0 \qquad \mbox{ in
}\,\,Q^+_T:=\R_+ \times [0,T),
\end{equation}
\begin{equation}\label{C20K-april9}
v(x,0)=v_0\quad \mbox{ and }\quad v(x,t)=0,\quad x\in\partial\R_+ = {\mathbb R}^{n-1}.
\end{equation}
We review a solution representation of
\eqref{C10K-april9}-\eqref{C20K-april9} formulated by S. Ukai (see
\cite{U}). Let $R = (R', R_n)$ and $S = (S_1, \cdots , S_{n-1})$ be
Riesz's operators in ${\mathbb R}^n$ and ${\mathbb  R}^{n-1}$,
respectively, that is,
\begin{align*}
R_i f(x) &= c_n \int_{{\mathbb R}^n} \frac{x_i -y_i}{|x -y|^{n+1}}
f(y) dy,
\qquad x \in {\mathbb R}^n, \quad  1 \leq i \leq n,\\
S_i f(x', x_n) &= c_{n-1} \int_{{\mathbb R}^{n-1}} \frac{x_i
-y_i}{|x' -y'|^{n}} f(y', x_n) dy',  \qquad x' \in {\mathbb
R}^{n-1}, \quad  1 \leq i \leq n-1.
\end{align*}
The functional operator $V_1$ and $V_2$ are defined by
\begin{equation}\label{CK300-april9}
V_1 v_0(x) =  -(S \cdot v_0(\cdot, x_n) )    (x') + v^n_0(x),\qquad
V_2 u_0(x) = v_0^{'} (x) +  ( S v_0^n (\cdot, x_n) )(x').
\end{equation}
Further, let $\ga$ be the restriction operator  from ${\mathbb
R}^n_+$ to ${\mathbb R}^{n-1}$, namely $\ga g = g|_{{\mathbb
R}^{n-1}}$. For a given function $g:{\mathbb R}^n_+\rightarrow
{\mathbb R} $ we define a functional operator $U$ by
\begin{align*}
U g(x) = r R^{'} \cdot S ( R^{'} \cdot S + R_n ) eg(x), \qquad x \in
{\mathbb R}^n_+,
\end{align*}
where $r$ be the restriction operator from ${\mathbb R}^n$ to
${\mathbb R}^n_+$
and  $e$ the zero extension operator from ${\mathbb R}^n_+$ over
${\mathbb R}^n$, i.e.
\begin{equation}\label{CK100-april9}
rf = f|_{{\mathbb R}^n_+}, \qquad
ef = \left\{\begin{array}{ll}
f, & \mbox{for} \quad x_n > 0,\\
0, & \mbox{for} \quad x_n < 0.
\end{array}
\right.
\end{equation}
We also define the  integral operators $D$ and $E(t) $by
\begin{align*}
D g(x) & = c_n\int_{{\mathbb R}^{n-1}} \frac{x_n}{(|x' -y'|^2 + x_n^2)^{\frac{n}2}}  g(y') dy',\\
E(t) g(x) & = \int_{{\mathbb R}^n_+} \Big(\Ga(x-y,t) - \Ga(x-y^*, t)
\Big) g(y) dy.
\end{align*}
Then, the solution of \eqref{C10K-april9}-\eqref{C20K-april9} is
represented by
\begin{equation}\label{ukai-formula-v}
v'(x,t) = E(t) V_2 u_0 - SU r E(t) e V_1 u_0,\qquad v^n(x,t) = U  r
E(t) e V_1 u_0 (x,t),
\end{equation}
\begin{equation}\label{ukai-formula-p}
p (x,t) = - D \ga \pa_n E(t) e V_1 u_0.
\end{equation}
We denote, for simplicity, by $X$ one of function spaces
$H^\be_p({\mathbb R}^n)$, $\dot H^\be_p({\mathbb R}^n)$,
$B^\be_p({\mathbb R}^n)$ and $\dot B^\be_p({\mathbb R}^n)$. We then
note that the following functionals are bounded operators:
\begin{equation}\label{rs}
R_i : X\rightarrow X,\,\,\,i=1,\cdots,n,\quad\quad S_i: X\rightarrow
X,\,\,\,i=1,\cdots,n-1.
\end{equation}
Similarly we mean by $Y_+$ one of $H^\be_{p,0} ({\mathbb R}^n_+)$,
$B^\be_{p,0} ({\mathbb R}^n_+)$, $\dot H^\be_{p,0} ({\mathbb
R}^n_+)$ and $\dot B^\be_{p,0} ({\mathbb R}^n_+)$. We also observe
that
\begin{align}\label{vs}
V_1: Y_+\rightarrow Y_+, \qquad V_2: Y_+\rightarrow Y_+, \qquad U :
Y_+\rightarrow Y_+
\end{align}
are are bounded operators, where we used that $e:
B_{q,0}^\al({\mathbb R}^n_+)\ri B_{q}^\al({\mathbb R}^n)$ or  $e:
\dot B_{q,0}^\al({\mathbb R}^n_+) \ri \dot B_{q}^\al({\mathbb R}^n)$
is a bounded operator.

In next proposition we consider the heat equation with an initial
data in Besov spaces in $\mathbb R^n$.

\begin{prop}\label{E's}
Let $\al\ge 0$. Suppose that $u_0 \in \dot
B_{pq}^{\al-\frac2q}({\mathbb R}^n)$. For each $(x,t)\in {\mathbb
R}^n \times (0, \infty)$ we define
\begin{align}\label{u-1}
u(x,t) := \left\{ \begin{array}{ll}
< u_0(\cdot), \Ga(x -\cdot,t)>, \qquad\;\;\;\textrm{if}\;\;  0\le \al <\frac2q ,\\
\\
\int_{{\mathbb R}^n} \Ga(x-y,t)  u_0(y)\;
dy,\qquad\textrm{if}\;\;\frac2q \leq \al.
\end{array}
\right.
\end{align}
Then $ u \in   \dot H^{\al, \frac12 \al}_{pq,*} ( {\mathbb R}^n \times (0, \infty) ) )$ and the
following estimate is satisfied:
\begin{align}\label{estimate 1}
\| u \|_{\dot {H}^{\al, \frac12 \al}_{pq,*} ( {\mathbb R}^n \times (0,\infty))}
\le c\, \| u_0\|_{\dot B_{pq}^{\al-\frac2q} ({\mathbb R}^n) }.
\end{align}
If $ u_0 \in B_{pq}^{\al-\frac2q} ({\mathbb R}^n)$, then $u \in
H^{\al,\frac12 \al}_{pq,*} ( {\mathbb R}^n \times (0,\infty) )$ and
\begin{align}\label{estimate 2}
\| u \|_{ {H}^{\al, \frac12 \al}_{pq,*} ( {\mathbb R}^n \times (0,\infty))} \le
c\, \| u_0\|_{ B_{pq}^{\al-\frac2q} ({\mathbb R}^n) }.
\end{align}
\end{prop}
Although Proposition \ref{E's} may be known in experts, we are not
able to find it in the literature and thus we provide the proof of
our own. We start with next lemma related to theory of multipliers.
\begin{lemm}\label{multiplier2}
Let $\Phi (\xi) = \hat \phi(2^{-1} \xi) + \hat \phi (\xi) + \hat
\phi (2\xi) $ with $\phi$ in the definition of Besov spaces. Suppose
that $\Phi_j (\xi) = \Phi(2^{-j} \xi)$ and $\rho_{tj}(\xi) = \Phi_j
( \xi) e^{-t|\xi|^2}$ for each integer $j$. Then, $\rho_{tj}( \xi)$
is a $L^p(\mathbb{R}^n)$-multiplier with the finite norm $M(t,j)$
for $1 < p <\infty$ such that for $t> 0$
\begin{equation*}
M(t,j) \le ce^{-\frac14 t2^{2 j}}\sum_{0 \leq i \leq n} t^i 2^{2 ij}
\le ce^{-\frac18 t2^{2 j}}.
\end{equation*}
\end{lemm}
\begin{proof}
We note that the $L^p(\mathbb{R}^n)$-multiplier norm $M(t,j)$ of
$\rho_{tj}(\xi) $ is equal to the $L^p(\mathbb{R}^n)$-multiplier
norm of $\rho_{tj}^{'}(\xi)  := \Phi(\xi) e^{-t2^{2 j} |\xi|^2}$
(see \cite[Theorem 6.1.3]{BL}). Let $l$ be an integer with $1\le
l\le N$ and $\be=(\be_1, \be_2, \cdots, \be_n)\in\mathbb R^n$.
Suppose $\{{i_1}, {i_2}, \cdots, {i_l}\}\subset\{1, 2,\cdots, n\}$
and assume $\be_{i_1}=\be_{i_2}=\cdots =\be_{i_l} =1$ and $\be_{i}
=0$ for $i\neq i_k$ with $k=1, \cdots, l$.
Since ${\rm supp}\;(\Phi) \subset \{\xi \in \mathbb{R}^n \, | \,
\frac14 < |\xi| < 4 \}$, we have
\begin{align*}
|D^\be_{\xi} \rho_{tj}^{'}(\xi)| &\lesssim   e^{-\frac14t 2^{2 j}}
\chi_{\{\frac14 < |\xi| < 4\}} (\xi) \sum_{0 \leq i \leq |\be|} t^i
2^{2 ij},
\end{align*}
where $\chi_A$ is the characteristic function on a set $A$. Hence,
for $A = \prod_{1 \leq i \leq l} [2^{k_i}, 2^{k_i +1}]$ we obtain
\begin{align*}
\int_A \left|D^\be_{\xi} \rho_{tj}^{'}(\xi)\right| d\xi_\be \leq c  e^{-\frac14t 2^{2 j}} \sum_{0 \leq i \leq n} t^i
2^{2 ij}.
\end{align*}
Due to Theorem 4.6\'{} in \cite{St}, we deduce the lemma.
\end{proof}
Note that if $g$ is smooth, then we have
\begin{align*}
J_{1 -\si} g(t) = \int_t^\infty (s -t)^{1 -\si} g'(s) ds,\qquad
D_t^*\si g(t) = \int_t^\infty \frac{g'(s)} {(s -t)^\si} ds.
\end{align*}
\begin{lemm}\label{multiplier3}
Suppose
that $\Phi_j (\xi) = \Phi(2^{-j} \xi)$ and $\te_{tj}(\xi) = \Phi_j
( \xi) |\xi|^2 \int_t^\infty \frac{ e^{-s|\xi|^2}} {(s -t)^\si} ds $ for each integer $j$. Then, $\te_{tj}( \xi)$
is a $L^p(\mathbb{R}^n)$-multiplier with the finite norm $N(t,j)$
for $1 < p <\infty$ such that for $t> 0$
\begin{equation*}
N(t,j) \le c  2^{2\si j} e^{-\frac14 t2^{2 j}}\sum_{0 \leq i \leq n} t^i 2^{2 ij}
\le c 2^{2\si j} e^{-\frac18 t2^{2 j}}.
\end{equation*}

\end{lemm}

\begin{proof}
We note that the $L^p(\mathbb{R}^n)$-multiplier norm $N(t,j)$ of
$\te_{tj}(\xi) $ is equal to the $L^p(\mathbb{R}^n)$-multiplier
norm of $\te_{tj}^{'}(\xi)  := \Phi(\xi) 2^{2j} |\xi|^2 \int_t^\infty \frac{ e^{-s2^{2j}|\xi|^2}} {(s -t)^\si} ds $
(see \cite[Theorem 6.1.3]{BL}). Let $l$ be an integer with $1\le
l\le N$ and $\be=(\be_1, \be_2, \cdots, \be_n)\in\mathbb R^n$.
Suppose $\{{i_1}, {i_2}, \cdots, {i_l}\}\subset\{1, 2,\cdots, n\}$
and assume $\be_{i_1}=\be_{i_2}=\cdots =\be_{i_l} =1$ and $\be_{i}
=0$ for $i\neq i_k$ with $k=1, \cdots, l$.
Since ${\rm supp}\;(\Phi) \subset \{\xi \in \mathbb{R}^n \, | \,
\frac14 < |\xi| < 4 \}$, we have
\begin{align*}
|D^\be_{\xi} \te_{tj}^{'}(\xi)| &\lesssim   2^{2j}  \chi_{\{\frac14
< |\xi| < 4\}} (\xi)  \sum_{0 \leq i \leq
|\be|}    \int_t^\infty \frac{ (s 2^{2j})^i e^{-\frac14 s2^{2j}}} {(s -t)^\si} ds\\
& =   2^{2j} e^{-\frac14 t2^{2j}}  \chi_{\{\frac14 < |\xi| < 4\}}
(\xi)  \sum_{0 \leq i \leq
|\be|}  2^{2ij}  \int_0^\infty \frac{ (s+t)^i e^{-\frac14 s2^{2j}}} {s^\si} ds\\
& =   2^{2j} e^{-\frac14 t2^{2j}}  \chi_{\{\frac14 < |\xi| < 4\}}
(\xi)  \sum_{0 \leq i \leq
|\be|}  2^{2ij} \sum_{0 \leq l \leq i}  t^{i-l}  \int_0^\infty  s^{l -\si} e^{-\frac14 s2^{2j}}  ds\\
& =   2^{2j} e^{-\frac14 t2^{2j}}  \chi_{\{\frac14 < |\xi| < 4\}}
(\xi)  \sum_{0 \leq i \leq
|\be|}  2^{2ij} \sum_{0 \leq l \leq i}  t^{i-l} 2^{-2j(l -\si +1)} \int_0^\infty  s^{l -\si} e^{-\frac14 s}  ds\\
& \lesssim    2^{2\si j} e^{-\frac14 t2^{2j}}  \chi_{\{\frac14 <
|\xi| < 4\}} (\xi)    \sum_{0   \leq i \leq n}  t^i2^{2ij}.
\end{align*}
Hence,
for $A = \prod_{1 \leq i \leq l} [2^{k_i}, 2^{k_i +1}]$ we obtain
\begin{align*}
\int_A \left|D^\be_{\xi} \te_{tj}^{'}(\xi) \right| d\xi_\be \leq c  2^{2\si j}  e^{-\frac14t 2^{2 j}} \sum_{0 \leq i \leq n} t^i
2^{2 ij}.
\end{align*}
Due to Theorem 4.6\'{} in \cite{St}, we deduce the lemma.
\end{proof}

We are ready to give the proof of Proposition \ref{E's}.

\begin{prop4}
We only treat the case that $ u_0\in \dot B^{k-\frac{2}{p}}
({\mathbb R}^n)$, since the proof of the case for $B^{k-\frac{2}{p}}
({\mathbb R}^n)$ is similar. We first show that for $k=0$
\begin{align}\label{boundary2}
\|v\|_{L_x^pL^q_t} \leq c \;\| u_0\|_{ {\dot B}^{-\frac{2}q}_{pq}
({\mathbb R}^n)}.
\end{align}
As usual, we may assume that $u_0  \in C^\infty_c ({\mathbb R}^n)$,
since $C^\infty_c({\mathbb R}^n)$ is dense in $ {\dot
B}^{-\frac{2}q}_{pq}({{\mathbb R}^{n-1}})$. Using the dyadic
partition of unity $\hat{\psi} (\xi) + \sum_{j=1}^{\infty} \hat
\phi(2^{-j} \xi) =1$ for $\xi \in {{\mathbb R}^{n-1}}$, we can write
\begin{align*}
\hat v(\xi,t) = 
\sum_{j=-\infty}^\infty \hat\phi (2^{-j} \xi) e^{-t|\xi |^2}
\widehat{ u_0}(\xi).
\end{align*}
For $t>0$ we have
\begin{equation}\label{u_1}
\int_{{\mathbb R}^n}  | v(x,t)|^pdx\leq \int_{{\mathbb R}^n}
\left|{\mathcal F}^{-1} \Big(\sum_{j=-\infty}^\infty e^{-t|\xi|^2}
\hat \phi_j(\xi) \;\widehat{ u_0}(\xi) \Big)(x)\right|^p dx.
\end{equation}
We note that $\hat \phi_j = \Phi_j \hat \phi_j$  for all $j$, where
$\Phi_j$ is defined in Lemma \ref{multiplier2} and by Lemma
\ref{multiplier2}, $\ \Phi_j( \xi) e^{-t |\xi|^2}$ is the
$L^p({\mathbb R}^n)$- multipliers with the norms $M(t,j)$.
Then we divide the sum as
\begin{align*}
&\int_{{\mathbb R}^n} \left|{\mathcal F}^{-1}
\Big(\sum_{j=-\infty}^\infty e^{-t|\xi|^2} \hat \phi_j(\xi)\;
\widehat{ u_0}(\xi) \Big)(x)\right|^p dx = \int_{{\mathbb R}^n}
\left|{\mathcal F}^{-1} \Big(\sum_{j=-\infty}^\infty \Phi_j( \xi)
e^{-t|\xi|^2} \hat \phi_j(\xi)\;\widehat{
u_0}(\xi) \Big)(x)\right|^p dx\\
& \leq \Big(\sum_{2^{2j}\le1/t} M(t,j) \|  u_0 * \phi_j\|_{L^p}
\Big)^p + \Big(\sum_{2^{2j}\ge1/t} M(t,j) \|  u_0 *\phi_j\|_{L^p}
\Big)^p=: I_1(t) + I_2(t).
\end{align*}
Here, ${\mathcal F}^{-1}$ is the inverse Fourier transform in $\R$.
By Lemma \ref{multiplier2} we have $ M(t,j) \leq c$  for $t 2^{2 j}
\leq 1$. We  take $a$ satisfying $ -\frac{2}q < a < 0$ and then use
H\"older inequality to get
\[
\int_0^\infty I_1(t)^{\frac{q}p} dt\lesssim \int_0^\infty
\Big(\sum_{2^{2j}\le1/t} 2^{-\frac{q}{q-1}aj} \Big)^{q-1}
\sum_{2^{2j}\le1/t} 2^{qaj} \|\phi_j* u_0\|^q_{L^p}dt
\]
\[
\lesssim  \int_0^\infty t^{\frac12 qa } \sum_{2^{2j}\le1/t} 2^{qaj}
\|\phi_j* u_0\|^q_{L^p}dt
\]
\[
\lesssim \sum_{j=-\infty}^\infty 2^{qaj} \|\phi_j* u_0\|^q_{L^p}
\int_0^{2^{-2 j}}  t^{\frac12 qa }dt= c \sum_{j=-\infty}^\infty
2^{-2 j } \|\phi_j* u_0\|^q_{L^p}.
\]
Using Lemma \ref{multiplier2} again, we have $ M(t,j) \leq c (t2^{2
j})^{-m} \sum_{0 \leq i \leq n}(t 2^{2 j})^i \leq c 2^{(2 n-2 m)j}
t^{n-m} $ for $t\cdot 2^{2 j} \geq 1$ and $m>0$. We fix $b>0$ and
then choose $m$ satisfying $ q\,(n-m) + \frac12 q\,b +1 < 0$, so
that we obtain
\[
\int_0^\infty I_2(t)^{\frac{q}p} dt \lesssim \int_0^T
\Big(\sum_{2^{2j}\ge1/t} 2^{(2 n- 2 m)j} t^{n-m} \|\phi_j*
u_0\|_{L^p} \Big)^qdt
\]
\[
\lesssim \int_0^\infty t^{q ( n -  m) } \Big(\sum_{2^{2j}\ge1/t}
2^{-\frac{q}{q-1} bj} \Big)^{q-1} \sum_{2^{2j}\ge1/t} 2^{qbj}2^{q(2
n-2 m)j} \|\phi_j* u_0\|^q_{L^p}dt
\]
\[
\lesssim \int_0^\infty t^{ q(  n- m )  + \frac12 qb}
\sum_{2^{2j}\ge1/t} 2^{qbj}2^{q(2 n-2 m)j} \|\phi_j* u_0\|^q_{L^p}dt
\]
\[
\lesssim \sum_{j=-\infty}^\infty 2^{qbj}2^{q(2 n-2 m)j} \|\phi_j*
u_0\|^q_{L^p} \int_{2^{-2j}}^\infty t^{q ( n - m )  + \frac12 qb}dt
=c \sum_{j=-\infty}^\infty 2^{-2 j  } \|\phi_j* u_0\|^q_{L^p}.
\]
Therefore, we obtain \eqref{boundary2}. For  $k  \in {\bf N}$ with
$k\geq 1$, we have
\begin{align*}
D^k_x v(x,t) = \int_{{\mathbb R}^n} \Ga(x-y, t) D^k_y u_0(y) dy.
\end{align*}
Using the estimate \eqref{boundary2}, we obtain
\begin{equation}\label{w-2}
\int_0^\infty  \|D^k_x v(\cdot,t)\|_{L^p({\mathbb R}^n)}^q dt \leq c\|D^k
u_0\|^q_{\dot B^{-\frac2q}_{pq}( {\mathbb R}^n)} \leq c\|
u_0\|^q_{\dot B^{k-\frac2q}_{pq}( {\mathbb R}^n)}.
\end{equation}
%
Next, we note that
\begin{align*}
\widehat{ D^{*\frac12 \al}_t v}(\xi,t) = 
\sum_{j=-\infty}^\infty \hat\phi (2^{-j} \xi)|\xi|^2 \int_t^\infty \frac{ e^{-s|\xi |^2}}{(s-t)^{\frac12 \al}} ds
\widehat{ u_0}(\xi).
\end{align*}
By the same argument, we get
\begin{align*}
\int_{{\mathbb R}^n}  | D^{*\frac12 \al}_t v(x,t)|^pdx
& \leq \Big(\sum_{2^{2j}\le1/t} N(t,j) \|  u_0 * \phi_j\|_{L^p}
\Big)^p + \Big(\sum_{2^{2j}\ge1/t} N(t,j) \|  u_0 *\phi_j\|_{L^p}
\Big)^p\\
&: = J_1(t) + J_2(t).
\end{align*}
Due to Lemma \ref{multiplier3}, we have $ N(t,j) \leq c 2^{\frac12
\al j} $  for $t 2^{2 j} \leq 1$. We  take $a$ satisfying $
-\frac{2}q < a < 0$ and then use H\"older inequality to get
\[
\int_0^\infty J_1(t)^{\frac{q}p} dt\lesssim \int_0^\infty
\Big(\sum_{2^{2j}\le1/t} 2^{-\frac{q}{q-1}aj} \Big)^{q-1}
\sum_{2^{2j}\le1/t} 2^{qaj +  q \al} \|\phi_j* u_0\|^q_{L^p}dt
\]
\[
\lesssim  \int_0^\infty t^{\frac12 qa } \sum_{2^{2j}\le1/t} 2^{qaj +  q \al}
\|\phi_j* u_0\|^q_{L^p}dt
\]
\[
\lesssim \sum_{j=-\infty}^\infty 2^{qaj +  q \al} \|\phi_j* u_0\|^q_{L^p}
\int_0^{2^{-2 j}}  t^{\frac12 qa }dt= c \sum_{j=-\infty}^\infty
2^{  q \al-2 j } \|\phi_j* u_0\|^q_{L^p}.
\]
Via Lemma \ref{multiplier3} we note that $ N(t,j) \leq c (t2^{2
j})^{-m} \sum_{0 \leq i \leq n}(t 2^{2 j})^i \leq c 2^{(2 n-2 m)j}
t^{n-m} $ for $t\cdot 2^{2 j} \geq 1$ and $m>0$. We fix $b>0$ and
then choose $m$ satisfying $ q\,(n-m) + \frac12 q\,b +1 < 0$, so
that we obtain
\[
\int_0^\infty J_2(t)^{\frac{q}p} dt \lesssim \int_0^\infty
\Big(\sum_{2^{2j}\ge1/t} 2^{(2 n- 2 m + \al)j} t^{n-m} \|\phi_j*
u_0\|_{L^p} \Big)^qdt
\]
\[
\lesssim \int_0^\infty t^{q ( n -  m) } \Big(\sum_{2^{2j}\ge1/t}
2^{-\frac{q}{q-1} bj} \Big)^{q-1} \sum_{2^{2j}\ge1/t} 2^{qbj}2^{q(2
n-2 m + \al)j} \|\phi_j* u_0\|^q_{L^p}dt
\]
\[
\lesssim \int_0^\infty t^{ q(  n- m )  + \frac12 qb}
\sum_{2^{2j}\ge1/t} 2^{qbj}2^{q(2 n-2 m +\al)j} \|\phi_j* u_0\|^q_{L^p}dt
\]
\[
\lesssim \sum_{j=-\infty}^\infty 2^{qbj}2^{q(2 n-2 m +\al)j} \|\phi_j*
u_0\|^q_{L^p} \int_{2^{-2j}}^\infty t^{q ( n - m )  + \frac12 qb}dt
=c \sum_{j=-\infty}^\infty 2^{\al j-2 j  } \|\phi_j* u_0\|^q_{L^p}.
\]
Using the estimate \eqref{boundary2}, we obtain
\begin{align}\label{w-2-2}
\int_0^\infty \| D_t^{*\frac12 \al} v(x, \cdot)\|_{L^p({\mathbb R}^n)}^q dx  \leq
c\| u_0\|^q_{\dot B^{\al-\frac2q}_{pq}( {\mathbb R}^n)}.
\end{align}
Hence, from \eqref{w-2} and \eqref{w-2-2}, we complete the proof of Proposition \ref{E's}.
\end{prop4}

Returning to Stokes system \eqref{C10K-april9} and
\eqref{C20K-april9}, via the solution formulae
\eqref{ukai-formula-v}-\eqref{ukai-formula-v} and the boundedness of
$V_1, V_2, U$ and $E$ (see \eqref{rs}-\eqref{vs} and Proposition
\ref{E's}), we have
\begin{align}\label{u-n}
\| v\|_{\dot H^{\al, \frac12 \al }_{pq} ({\mathbb R}^n_+ \times
(0,T))} &\leq c \|    u_0 \|_{\dot B^{\al -\frac2p}_{pq,0} ({\mathbb
R}^n_+)}.
\end{align}
For the estimate of  pressure, since $D: \dot B^{k -\frac1p}_{p}
({\mathbb R}^{n-1}) \ri \dot H^{k }_p ({\mathbb R}^n_+)$ with $k
\geq 0$ is a bounded operator, respectively, we obtain that for $\al
> 1 + \frac1p$,
\[
\int_0^T\| p(\cdot, t)\|^q_{\dot H^{\al -1  }_p ({\mathbb R}^n_+
 )} dt=  \int_0^T \| D \ga \pa_n E(t) V_1 u_0\|^q_{\dot H_p^{\al -1} ({\mathbb R}^n_+  )} dt
\]
\[
\leq  c \int_0^T \|   \pa_n E(t) V_1 u_0\|^q_{\dot H_p^{\al -1}
({\mathbb R}^n_+)} dt
\]
\begin{equation}\label{CK100-april10}
\leq c \int_0^T \|    E(t) V_1 u_0\|^q_{\dot H^{\al }_p ({\mathbb
R}^n_+  )}  dt \leq c\| u_0\|_{\dot B^{\al -\frac2q }_{pq,0}
({\mathbb R}^n_+ )}.
\end{equation}
Combining \eqref{u-n}-\eqref{CK100-april10}, we conclude that
$(v,p)$ of Stokes system \eqref{C10K-april9} and \eqref{C20K-april9}
satisfies the following estimate:
\begin{equation}\label{CK110-april10}
\| v\|_{\dot H^{\al, \frac12 \al }_{pq} ({\mathbb R}^n_+ \times
(0,T))}+  \| p\|_{L^q ((0,T); \dot H^{\al -1 }_p({\mathbb R}^n_+))}
\leq c \|    u_0 \|_{\dot B^{\al -\frac2q}_{pq,0} ({\mathbb
R}^n_+)}.
\end{equation}
As mentioned earlier, following similar procedures as above, we can
prove that
\begin{equation}\label{CK120-april10}
\| v\|_{H^{\al, \frac12 \al }_p ({\mathbb R}^n_+ \times (0,T))}+ \|
p\|_{L^q((0,T); H^{\al -1}_p ({\mathbb R}^n_+ ))}  \leq c \|    u_0
\|_{B^{\al -\frac2q}_{pq,0} ({\mathbb R}^n_+)}.
\end{equation}
Since the proof of \eqref{CK120-april10} is on the same track, we
skip its details. Summing up above results, the proof of Theorem
\ref{maintheo-stokes} is fulfilled.

\begin{proof-stokes}
Since the solution of \eqref{maineq2}-\eqref{maineq2-10} is the sum
of solutions for \eqref{C10K-april7}-\eqref{C20K-april7} and
\eqref{C10K-april9}-\eqref{C20K-april9}, the proof can be done by
combining estimates \eqref{C110K-april7}-\eqref{C200K-april10} and
\eqref{CK110-april10}-\eqref{CK120-april10} established in section 3
and Section 4. Uniqueness of solution can be established by using
duality considerations by following the argument in \cite[Theorem
2.1]{So2}. More precisely, let $\tilde{v}$ and $\hat{v}$ be two
solutions in the sense of distributions \eqref{CK-july24-100}. We
set $w:=\tilde{v} - \hat{v}$, which is certainly in $L^q (
(0,T);L^p(\R_+))$ and we then observe that for any $T' < T$ the
following equality is satisfied:
\[
\int^{T'}_0\int_{{\mathbb R}^n_+}w\cdot(-\phi_t-\Delta \phi+\nabla
\pi)dxdt=0
\]
for any compactly supported smooth vector $\phi$ with $\phi|_{t =T'}
=0$ and smooth scalar function $\pi$. Since the vector field
$-\phi_t-\De \phi+\na \pi$ is dense in $L^{q'} (
(0,T);L^{p'}(\R_+))$, it follow that $w=0$. This completes the
proof.
\end{proof-stokes}


\section{Proofs of Theorem \ref{maintheo-stokes-g} and Theorem \ref{main-nse-est}}

In this section, we present the proofs of Theorem
\ref{maintheo-stokes-g} and Theorem \ref{main-nse-est}. We begin
with Theorem \ref{maintheo-stokes-g}, which is a direct consequence
of Theorem \ref{maintheo-stokes}.

\begin{proof-stokes-g}
\quad In Theorem \ref{maintheo-stokes-g}, Since $F\in
L^q(0,T;H^{\beta}_{p,0}(\R_+))$, we note that $\nabla\cdot F$
belongs to $L^q(0,T;H^{-1+\beta}_{p,0}(\R_+)$ and
\begin{equation}\label{CK-june12-100}
\norm{\nabla\cdot F}_{L^q(0,T;H^{-1+\beta}_{p,0}(\R_+)}\leq
c\norm{F}_{L_q(0,T;H^{\beta}_{p,0}(\R_+))}.
\end{equation}
On the other hand, due to the result of Theorem
\ref{maintheo-stokes-g}, we obtain
\begin{equation}\label{CK-june12-200}
\| v\|_{H^{1+\beta, \frac{1+\beta}{2}}_{pq}(\R_+ \times (0,T))} \leq
c \|\nabla\cdot F\|_{H^{-1+\beta, \frac{-1+\beta}{2}}_{pq,0}
({\mathbb R}^n_+ \times (0,T))}+c \|
v_0\|_{{B}^{1+\beta-\frac{2}{q}}_{pq,0}(\R_+)},\qquad 0\le\beta\le
1.
\end{equation}
\begin{equation}\label{CK-june12-300}
\| p\|_{L_q(0,T; H^{\beta}_p(\R_+))}\leq c \| \nabla\cdot
F\|_{L_q(0,T;H^{-1+\beta}_{p,0}(\R_+))}+c \|
v_0\|_{{B}^{1+\beta-\frac{2}{q}}_{pq,0}(\R_+)},\qquad
\frac{1}{p}<\beta\le 1.
\end{equation}
We observe via the estimate \eqref{CK-june12-100} that
\begin{equation}\label{CK-june12-400}
\|\nabla\cdot F\|_{H^{-1+\beta, \frac{-1+\beta}{2}}_{pq,0} ({\mathbb
R}^n_+ \times (0,T))}\leq \norm{\nabla\cdot
F}_{L^q(0,T;H^{-1+\beta}_{p,0}(\R_+)}\leq
c\norm{F}_{L_q(0,T;H^{\beta}_{p,0}(\R_+))}.
\end{equation}
Combining \eqref{CK-june12-200}-\eqref{CK-june12-400}, the estimates
\eqref{CK-est-300} and \eqref{CK-est-310} are immediate.
\end{proof-stokes-g}

We first recall the definition of weak solutions for the
Navier-Stokes equations.
\begin{defin}\label{weak-def}
Let $v_0 \in L^2(\Omega)$ with ${\rm{div}} \, v_0 = 0$. We say $v$
is a weak solution of \eqref{nse-10-march22}-\eqref{nse-20-march22}
if $v$ satisfies the following:
\begin{itemize}
\item[(i)]  $v \in L^{\infty}(0,T;L^2(\Omega)) \cap L^2(0,T
;H^1_0(\Omega))$ and $v$ satisfies
\begin{equation*}
\int_0^{T}\int_{\Omega}\bke{\frac{\partial\phi}{\partial
t}+(u\cdot\nabla)\phi}udxdt+\int_{\Omega}u_0\phi(x,0)dx
=\int_0^{T}\int_{\Omega}\nabla u :\nabla\phi dxdt
\end{equation*} for all $\phi\in
\calC_0^{\infty}(\Omega\times [0,T))$ with ${\rm{div}} \, \phi=0$.
\item[(ii)] $v$ satisfies divergence free condition, that is,
$\int_{\Omega} u\cdot\nabla\psi dx=0$ for any
$\psi\in\calC^{\infty}(\bar{\Omega})$.
\end{itemize}
\end{defin}

Next, we decompose the nonlinear term of Helmholtz type such that
the result of Theorem \ref{maintheo-stokes} can be applicable.
%
\begin{lemm}\label{decompositon-helmholtz}
Suppose that the assumptions in Theorem \ref{main-nse-est} hold.
Then the nonlinear term $(v \cdot \nabla ) v$ is decomposed as ${\rm
div}\,(v\otimes v)=-\nabla \pi+ w$, where $w$ has zero divergence
and $\pi$ solves in a weak sense
\begin{equation}\label{nse-60}
\left\{
\begin{array}{cl}
\displaystyle-\Delta \pi=\partial_{x_i}\partial_{x_j}(v_i v_j)
&\qquad\mbox{ in }\,\, \mathbb{R}^3_+,
\\
\vspace{-3mm}\\
\displaystyle \pi=0&\qquad\mbox{ on }\,\, \partial\mathbb{R}^3_+
\end{array}\right.
\end{equation}
such that the following estimate is satisfied: for almost all $t\in
I$
\begin{equation}\label{CK-june13-100}
\norm{\pi(\cdot,t)}_{\dot{H}^{\beta}_{p}({\mathbb{R}}^3_+)}+\norm{w(\cdot,t)}_{\dot{H}^{\beta-1}_{p,0}({\mathbb{R}}^3_+)
}\leq C\norm{(v\otimes
v)(\cdot,t)}_{\dot{H}^{\beta}_{p,0}(\mathbb{R}^3_+)}.
\end{equation}
\end{lemm}

Since standard theories of elliptic equations imply that
$\norm{\pi}_{\dot{H}^{\beta}_{p}({\mathbb{R}}^3_+)}\leq
C\norm{(v\otimes
v)(\cdot,t)}_{\dot{H}^{\beta}_{p,0}(\mathbb{R}^3_+)}$ (see also
Lemma \ref{lemma-2}) and $w=\nabla \pi+{\rm div}\,(v\otimes v)$, the
estimate \eqref{CK-june13-100} is immediate, and therefore we skip
the details of Lemma \ref{decompositon-helmholtz}. Next lemma shows
an estimate of interpolation for fractional derivative of square
functions.
\begin{lemm}\label{interpolation-june13}
Let $0 < \be \le 1$, $1 \leq p < \infty$ and $1/{p_1} + 1/{p_2}
=1/p$ with $1<p_2<\infty$.
%
%
Then, the following a priori estimate holds:
\begin{equation}\label{CK-june14-20}
\| u^2\|_{ H^\be_p(\R_+)} \leq c \Big( \| u^2\|_{L^p(\R)} +  \|
u\|_{L^{p_1} (\R_+)}\| u\|^{1-\be}_{L^{p_2}(\R_+)} \|
u\|_{H^1_{p_2}(\R_+)}^\be +  \| u\|^{2-\be}_{L^2(\R_+)} \|
u\|_{H^1_2(\R_+)}^\be  \Big).
\end{equation}
\end{lemm}
\begin{proof}
We note that
\begin{align}\label{u-2betanorm}
\| u^2\|_{ H^\be_p(\R_+)} \approx \| u^2\|_{ H^\be_p(\R)} \approx \|
u^2\|_{L^p(\R)} + \| \De^{\frac{\be}{2}} u^2 \|_{L^p(\R)},
\end{align}
where
\begin{align*}
\De^{\frac{\be}2} f(x) =- c(n, \beta) \int_{\R} \frac{f(x) -f(y)}{|x
-y|^{n +\be}} dy
\end{align*}
with a positive constant $c(n, \beta)$. Simple computations show
that
\begin{equation}\label{CK-june14-10}
\De^{\frac{\be}2} u^2(x) = 2u(x) \De^{\frac{\be}2} u(x) + c(n, \al)
 \int_{\R} \frac{(u(x) +u(y))^2}{|x -y|^{n +\be}} dy.
\end{equation}
Hence, we obtain
\begin{align*}
\|u \De^{\frac{\be}2} u\|_{L^p} \leq \| u\|_{L^{p_1}}
\|\De^{\frac{\be}2} u\|_{L^{p_2}} \leq \| u\|_{L^{p_1}} \|
u\|_{H^\be_{p_2}} .
\end{align*}
We recall that  for $1 \leq r < \infty$ (see e.g. \cite[Theorem
6.4.4]{BL})
\begin{align*}
 \| u\|_{ H^{\beta}_{r}(\mathbb{R}^n)} \leq \left\{
\begin{array}{l}
\|u\|_{B^\be_r(\R)}, \qquad 1 \leq r \leq 2,\\
\vspace{-3mm}\\
\|u\|_{B^\be_{r2}(\R)}, \qquad 2 \leq r < \infty.
\end{array}
\right.
\end{align*}
Via interpolations $(L^r(\R), H^1_r(\R))_{\be} = {B^\be_r(\R)}$ and
$(L^r(\R), H^1_r(\R))_{\be, 2} = {B^\be_{r2}(\R)}$ (refer to
\cite[section 3.5]{BL}), we have $\| u\|_{
H^{\beta}_{p_2}(\mathbb{R}^n)} \leq c \| u\|^{1-\be}_{L^{p_2}(\R_+)}
\| u\|_{H^1_{p_2}(\R_+)}^\be$ and therefore,
\begin{align}\label{firstfractional}
\|u \De^{\frac{\be}2} u\|_{L^{p}(\R)} \leq c\| u\|_{L^{p_1}(\R_+)}\|
u\|^{1-\be}_{L^{p_2}(\R_+)} \| u\|_{H^1_{p_2}(\R_+)}^\be .
\end{align}
Using similar arguments for the second term in \eqref{CK-june14-10},
we get
\[
(\int_{\R} (\int_{\R} \frac{|u(x) -u(y)|^2}{|x -y|^{n +\be}} dy
)^pdx)^{\frac{1}{p}} =
 (\int_{\R} (\int_{\R} \frac{|u(x) -u(y)|^2}{|x
-y|^{n +2 \frac{\be}2}} dy )^{\frac{2p}{2}}dx)^{\frac{2}{2p}}
\]
\begin{equation}\label{secondfractional}
\leq  \| u\|^2_{B^{\frac{\be}2}_{2, 2p}(\R)}\leq c \|
u\|^{2-\be}_{L^2(\R_+)} \| u\|_{H^1_2(\R_+)}^\be .
\end{equation}
From \eqref{u-2betanorm}, \eqref{firstfractional} and
\eqref{secondfractional}, we deduce the lemma.
\end{proof}

Next we give the proof of Theorem \ref{main-nse-est}.
\begin{proof-nse-a}
Let $\hat{p}=p-\pi$, where $\pi$ is given in \eqref{nse-60}. We then
obtain by the result of Theorem \ref{maintheo-stokes} that
\[
\| v\|_{L^q(0,T; H^{1+\beta}_p(\mathbb{R}^3_+))} +\|
\hat{p}\|_{L^q(0,T; H^{\beta}_p(\mathbb{R}^3_+))} \leq c \|
w\|_{L^q(0,T; H^{\beta-1}_{p,0}(\mathbb{R}^3_+ )) }+c \| v_0\|_{
{B}^{1+\beta-\frac{2}{q}}_{pq,0}(\mathbb{R}^3_+) },
\]
where $w$ is given in Lemma \ref{decompositon-helmholtz}. Using the
estimate \eqref{CK-june13-100} after integrating it in time
variable, we have \eqref{nse-120-march22}.
It remains to show that the righthand side in
\eqref{nse-120-march22} is bounded. Due to the estimate
\eqref{CK-june14-20} by choosing $p_2=2$, we note that
\[
\norm{v^2}_{H^{\beta}_{p,0}(\mathbb{R}^3_+)}\leq C\Big( \|
v^2\|_{L^p(\mathbb{R}^3_+)} +  \| v\|^{2-\be}_{L^2(\mathbb{R}^3_+)}
\| v\|_{H^1_2(\mathbb{R}^3_+)}^\be +  \| v\|_{L^{p_1}
(\mathbb{R}^3_+)}\| v\|^{1-\be}_{L^{2}(\mathbb{R}^3_+)} \|
v\|_{H^1_{2}(\mathbb{R}^3_+)}^\be  \Big),
\]
where $1/p_1=1/p-1/2$. Taking $L^q-$norm in time variable, we obtain
\[
\norm{v^2}_{L^q(0,T; H^{\beta}_{p,0}(\mathbb{R}^3_+))}\leq
C\norm{v}^2_{L^{2q}(0,T; L^{2p}(\mathbb{R}^3_+))} +
\norm{v}^{2-\beta}_{L^{\infty}(0,T;
L^{2}(\mathbb{R}^3_+))}\norm{v}^{\beta}_{L^{2}(0,T;
H^{1}_{2}(\mathbb{R}^3_+))}
\]
\begin{equation}\label{CK-june14-30}
+C\norm{v}_{L^{\frac{2q}{2-q\beta}}(0,T;
L^{p_1}(\mathbb{R}^3_+))}\norm{v}^{1-\beta}_{L^{\infty}(0,T;
L^{2}(\mathbb{R}^3_+))}\norm{v}^{\beta}_{L^{2}(0,T;
H^{1}_{2}(\mathbb{R}^3_+))}.
\end{equation}
Since $2<p_1\le 6$ and $3/p_1+(2-q\beta)/q=3/2$, the last term in
\eqref{CK-june14-30} is finite. This completes the proof.
\end{proof-nse-a}

\section*{Acknowledgements}
T.-K. Chang's work was partially supported by NRF-2009-0088692
and K. Kang's work was partially supported by NRF-2012R1A1A2001373.

\begin{equation*}
\left.
\begin{array}{cc}
{\mbox{Tongkeun Chang}}\qquad&\qquad {\mbox{Kyungkeun Kang}}\\
{\mbox{Department of Mathematics }}\qquad&\qquad
 {\mbox{Department of Mathematics}} \\
{\mbox{Yonsei University
}}\qquad&\qquad{\mbox{Yonsei University}}\\
{\mbox{Seoul, Republic of Korea}}\qquad&\qquad{\mbox{Seoul, Republic of Korea}}\\
{\mbox{chang7357@yonsei.ac.kr }}\qquad&\qquad
{\mbox{kkang@yonsei.ac.kr }}
\end{array}\right.
\end{equation*}

\end{document}